\documentclass[12pt]{amsart}
\usepackage{amssymb,latexsym}
\usepackage{enumerate}
\usepackage{amsmath}
\usepackage{amsfonts}
\usepackage{latexsym}
\usepackage{color}
\usepackage{graphicx}

\def\OL{\relax\ifmmode {\sf L}\else{\textsf L}\fi}
\def\OR{\relax\ifmmode {\sf R}\else{\textsf R}\fi}

\newcommand{\be}{\begin{equation}}
\newcommand{\en}{\end{equation}}

\newcommand{\bea}{\begin{eqnarray}}
\newcommand{\ena}{\end{eqnarray}}

\newcommand{\beano}{\begin{eqnarray*}}
\newcommand{\enano}{\end{eqnarray*}}

\newcommand{\bee}{\begin{enumerate}}
\newcommand{\ene}{\end{enumerate}}

\newcommand{\bei}{\begin{itemize}}
\newcommand{\eni}{\end{itemize}}

\newcommand{\ad}{^{\mbox{\scriptsize $\dag$}}}
\newcommand{\mult}{\,{\scriptstyle \square}\,}
\newcommand{\vp}{\varphi}
\newcommand{\mc}{\mathcal}
\newcommand{\mb}{\mathbb}
\newcommand{\A}{\mathfrak{A}}
\newcommand{\Ao}{\mathfrak{A}_0}
\newcommand{\R}{R \!\!\!\! R}

\newcommand{\Hil}{{\mc H}}

\newcommand{\I}{{\mc I}}

\newcommand{\Lc}{{\mc L}}
\newcommand{\LL}{{\mc L}}

\newcommand{\D}{{\mc D}}

\newcommand{\M}{{\mc M}}
\newcommand{\BB}{{\mathfrak B}}
\newcommand{\K}{{\mathfrak K}}

\newcommand{\pppb}{{\P_{\BB}(\A)}}
\def\NG{{\mathfrak N}}
\def\L{{\mc L}}
\newcommand{\up}{\upharpoonright}

\def\H{{\mathcal H}}
\def\P{{\mathcal P}}
\def\R{{\mathcal R}}
\newtheorem{defn}{Definition}[section]
\newtheorem{thm}[defn]{Theorem}
\newtheorem{prop}[defn]{Proposition}
\newtheorem{lemma}[defn]{Lemma}
\newtheorem{cor}[defn]{Corollary}
\newtheorem{example}[defn]{Example}
\newtheorem{examples}[defn]{Examples}
\newtheorem{rem}[defn]{Remark}
\def\x{\relax\ifmmode {\mbox{*}}\else*\fi}
\newcommand{\beex}{\begin{example}$\!\!${\bf }$\;$\rm }
\newcommand{\enex}{ \end{example}}
\newcommand{\beexs}{\begin{examples}$\!\!${\bf }$\;$\rm }
\newcommand{\enexs}{ \end{examples}}
\newcommand{\berem}{\begin{rem}$\!\!${\bf }$\;$\rm }
\newcommand{\enrem}{ \end{rem}}
\newcommand{\bedefi}{\begin{defn}$\!\!${\bf }$\;$\rm }
\newcommand{\findefi}{\end{defn}}
\newcommand{\becor}{\begin{cor}$\!\!\!$  }
\newcommand{\encor}{\end{cor}}
\newcommand{\rcab}{{\mc R}_c(\A,\BB)}

\newcommand{\ha}{^{\ast}}

\newcommand{\ip}[2]{\langle {#1}|{#2}\rangle}

\newcommand{\LD}{\Lc^\dagger(\D)}

\newcommand{\LDH}{{\LL}\ad(\D,\H)}
\newcommand{\LDO}[1]{{\LL}\ad(#1)}
\newcommand{\w}{{\rm w}}

\newcommand{\rep}{{\rm Rep}_c(\A)}
\newcommand{\repb}{{\rm Rep}_c(\A,\BB)}
\newcommand{\pa}{partial \mbox{*-algebra}}
\newcommand{\das}{^{\dag {\rm\textstyle *}}}
\newcommand{\LBDH}{{\L}_b\ad(\D,\H)}
\newcommand{\LBD}{{\L}_b\ad(\D)}
\def\MM{{\mathfrak M}}
\newcommand{\noi}{\noindent}
\newcommand{\po}{partial O\mbox{*-algebra}}

\newcommand{\wmult}{{\scriptscriptstyle \Box}}
\renewcommand{\leq}{\leqslant}
\renewcommand{\geq}{\geqslant}

\begin{document}
\title[Fully representable and *-semisimple partial *-algebras]
{Fully representable and *-semisimple topological partial *-algebras}
\date{\today}

\author{J-P. Antoine}
\address{Institut de Recherche en Math\'ematique et Physique \\
Universit\'e Catholique de Louvain\\
B-1348   Louvain-la-Neuve\\
Belgium}
\email{jean-pierre.antoine@uclouvain.be}

\author{G. Bellomonte}
\address{Dipartimento di Matematica e Informatica,
Universit\`a di Palermo, I-90123 Palermo, Italy}
\email{bellomonte@math.unipa.it}

\author{C. Trapani}
\address{Dipartimento di Matematica e Informatica,
Universit\`a di Palermo, I-90123 Palermo, Italy}
\email{trapani@unipa.it}

\begin{abstract}
We continue our study of topological partial *-algebras,  focusing our attention to *-semisimple partial *-algebras, that is,
those that possess   a  {multiplication core} and sufficiently many *-representations.
We discuss the respective roles of   invariant  positive sesquilinear (ips) forms and representable continuous linear functionals  and focus on the case where the two notions
are completely interchangeable (fully representable \pa s) with the scope of characterizing a  *-semisimple partial *-algebra. 
Finally we describe various notions of boun\-ded elements in such a \pa, in particular, those defined in terms of a positive cone (order bounded elements).
The outcome is that, for an appropriate order relation,
one recovers the $\M$-bounded elements introduced in previous works.
\end{abstract}
\maketitle

\section{Introduction}\label{sect_intro}

Studies on partial *-algebras have provided so far a considerable amount of information about their representation
theory and their structure. Many results have been obtained for concrete  partial *-algebras,
i.e., partial *-algebras of closable operators (the so-called partial O*-algebras), but  a substantial body of knowledge has been
gathered also for abstract \pa s. A full analysis  has
been developed by Inoue and two of us some time ago and it can be found in the monograph \cite{ait_book}, where earlier articles are quoted.

In a recent paper \cite{att_2010}, we have started the analysis of certain types of bounded elements in a  partial *-algebra  $\A$  and their incidence on the representation theory of $\A$.
It was shown, in particular, that the crucial condition is that $\A$ possesses sufficiently many invariant  positive sesquilinear forms (ips-forms). The latter, in turn, generate
  *-representations, that is, *-homomorphisms into a partial O*-algebra, via the well-known GNS construction. As in the particular case of a \po, a spectral theory can  then be developed,
provided the \pa\ has sufficiently many bounded elements. To that effect, we have introduced in \cite{att_2010} the notion of \emph{$\M$-bounded elements}, associated to a sufficiently large  family  $\M$ of ips-forms.

We continue this study in the present work, focusing on topological \pa s that possess what we call a \emph{multiplication core}, that is, a dense subset of universal right multipliers with
all the regularity properties necessary for a decent representation theory. In particular, we will require that our \pa\
has sufficiently many *-representations, a property usually characterized, for topological *-algebras, in terms of
the so-called \emph{*-radical}. When the latter is reduced to $\{0\}$, the \pa\ is called \emph{*-semisimple}, the main subject of the paper. According to what we just said,  *-semisimplicity
is defined in terms of a family  $\M$ of ips-forms. Since it may be difficult to identify such a family in practice, we examine in what sense ips-forms  may be replaced by a special class of
continuous linear functionals, called \emph{representable}. {  This leads to identify a class of topological \pa s for which representable linear functionals and
  { ips-forms can be freely   replaced by one another}, since every representable linear functional comes (as for *-algebras with unit) from an ips-form.
These \pa s are called \emph{fully representable}
 (extending the analogous concept discussed in \cite{FTT} for locally convex quasi *-algebras) and the interplay of this notion with *-semisimplicity is investigated.}

This being done, we may come back to bounded elements of a *-semisim\-ple \pa,
 more precisely to elements bounded with respect to some positive cone, thus defined in purely algebraic terms. Early work in that direction   has been done by Vidav \cite{vidav} and  Schm\"udgen \cite{schm_weyl},  then generalized in our previous paper  \cite{att_2010}.
Here we consider several types of order on a \pa\ and analyze the corresponding notion of order bounded elements.
The outcome is that, under appropriate conditions, the correct notion  reduces to that of $\M$-bounded ones introduced in \cite{att_2010}.
Therefore, when the \pa\ has sufficiently many such elements, the whole spectral theory developed  in \cite{antratsc} and   \cite{att_2010} can be recovered.

The paper is organized as follows.  Section \ref{sect_prelim}  is devoted to some preliminaries  about \pa s, taken mostly from \cite{ait_book} and \cite{antratsc, att_2010}.
In addition, we introduce the notion of multiplication core and draw some consequences.
We introduce  in Section \ref{sect-suffmanyrep} the notion of *-semisimple \pa\ and discuss some of its properties. In Section \ref{sect-represfunct}, we compare the respective roles of
ips-forms and representable linear functionals, with particular reference to fully representable \pa s, and discuss the relationship of the latter notion with that of *-semisimple \pa. Finally, Section \ref{sect-bddelem} is devoted to the various notions of bounded elements, from
$\M$-bounded to order bounded ones.
\medskip

\section{Preliminaries}\label{sect_prelim}

\medskip The following preliminary definitions will be needed in the sequel. For more details we  {  refer to \cite{ait_book, schmu}.}

A \pa\ $\A$ is a complex vector space with conjugate linear involution  $\ha $ and a distributive partial multiplication
$\cdot$, defined on a subset $\Gamma \subset \A \times \A$, satisfying the property that $(x,y)\in \Gamma$ if, and only if,
$(y\ha ,x\ha )\in   \Gamma$ and $(x\cdot y)\ha = y\ha \cdot x\ha $.  From now on we will write simply $xy$ instead of $x\cdot y$ whenever
$(x,y)\in \Gamma$. For every $y \in \A$, the set of left (resp. right) multipliers of $y$ is denoted by $L(y)$ (resp. $R(y)$), i.e.,
$L(y)=\{x\in \A:\, (x,y)\in \Gamma\}$ (resp. $R(y)=\{x\in \A:\, (y,x)\in \Gamma\}$). We denote by $L\A$ (resp. $R\A$)  the space of universal left (resp. right) multipliers of $\A$.

In general, a \pa\ is not associative, but in several situations a weaker form of associativity holds. More precisely, we say
that $\A$ is \emph{semi-associative} if $y \in R(x)$ implies $yz\in R(x)$, for every $z \in R\A$, and
 $$
(xy)z=x(yz).
$$
The partial *-algebra $\A$ has a  unit if  there exists an element $e\in \A$ such that $e=e\ha$, $e\in R\A\cap L\A$ and $xe=ex=x$, for every $x\in \A$.

Let $\H$ be a complex Hilbert space and $\D$ a dense subspace of $\H$.
 We denote by $ \L\ad(\D,\H) $
the set of all (closable) linear operators $X$ such that $ {D}(X) = {\D},\; {D}(X\x) \supseteq {\D}.$ The set $
\L\ad(\D,\H ) $ is a  \pa\
 with respect to the following operations: the usual sum $X_1 + X_2 $,
the scalar multiplication $\lambda X$, the involution $ X \mapsto X\ad := X\x \up {\D}$ and the \emph{(weak)}
partial multiplication $X_1 \mult X_2 := {X_1}\ad\x X_2$, defined whenever $X_2$ is a weak right multiplier of
$X_1$ (we shall write $X_2 \in R^{\rm w}(X_1)$ or $X_1 \in L^{\rm w}(X_2)$), that is, whenever $ X_2 {\D} \subset
{\D}({X_1}\ad\x)$ and  $ X_1\x {\D} \subset {\D}(X_2\x).$

It is easy to check that $X_1 \in L^{\rm w}(X_2)$ if and only if there exists $Z \in \LDH$ such that
\begin{equation} \label{altwp}
\ip{X_2\xi}{X_1\ad \eta} = \ip{Z\xi}{\eta}, \quad \forall \xi, \eta \in \D.
\end{equation}
In this case $Z= X_1 \mult X_2$.
  $\LDH$ is neither associative nor semi-associative.
If $I$ denotes the identity operator of $\H$,   $I_\D:=I\up\D$  is the unit of the partial *-algebra $ \L\ad(\D,\H)$.

If $\NG \subseteq \LDH$ we denote by $R^{\w}\NG$ the set of right multipliers of all elements of $\NG$. We recall that
 \be\label{eq:RLDH}
R\LDH \equiv R^{\rm w}\LDH= \{A\in \LDH: A\mbox{ bounded and } A:\D \to \D\ha \},
\en
 where $$
\D\ha =\bigcap_{X\in \LDH}D({X\das}).
$$

We denote by $\LBDH$ the bounded part of $\LDH$, i.e., $\LBDH=\{X \in \LDH : X$  is a bounded operator$\}=\{X \in \LDH : \overline{X}\in {\mc B}(\H)\}$.

{A $\ad $-invariant} subspace $\MM$ of $\LDH$ is called a \emph{(weak) partial O*-algebra} if $X\mult Y \in \MM$, for every
$X, Y \in \MM$ such that $X \in L^{\rm w}(Y)$. $\LDH$ is the maximal partial O*-algebra on $\D$.

 The set $\LD:=\{X\in \LDH:\, X, X\ad:\D \to \D$\} is a *-algebra; more precisely, it is the maximal O*-algebra on $\D$
(for the theory of O*-algebras and their representations we refer to \cite{schmu}).

\medskip
In the sequel, we will need the following topologies on $\LDH$:
\bei
\item The \emph{strong topology} ${\sf t}_s$ on $\LDH$, defined by the seminorms
$$ p_\xi(X)=\|X\xi\|, \quad X \in \LDH, \, \xi \in \D.$$

\item  The \emph{strong* topology} ${\sf t}_{s^\ast}$ on  $\LDH$, defined by the seminorms
 $$p^*_\xi (X)= \max\{\|X\xi\|, \|X\ad\xi\|\}, \,  \xi \in \D.$$
\eni

\vspace{2mm} A \emph{*-representation} of a  \pa\ $\A$ in the
Hilbert space $\H$ is a linear map $\pi : \A \rightarrow\L\ad(\D(\pi),\H)$     such that:
(i) $\pi(x\x) = \pi(x)\ad$ for every $x \in \A$; (ii) $x \in L(y)$
in $\A$ implies $\pi(x) \in L^{\rm w}(\pi(y))$ and $\pi(x) \mult\pi(y) = \pi(xy).$ The subspace $\D(\pi)$ is called the \emph{domain} of the *-representation $\pi$.
 The *-repres\-entation  $\pi$ is said to be \emph{bounded} if $\overline{\pi (x)} \in {\mc B}(\H)$ for every $x \in\A$.

\vspace{2mm} Let $\varphi$ be a positive sesquilinear form on
$D(\varphi) \times D(\varphi)$, where $D(\varphi)$ is a
subspace of $\A$. Then we have
\begin{align}
\varphi(x,y) &= \overline{\varphi(y,x)}, \ \ \ \forall \, x, y \in
D(\varphi),
\\
 |\varphi(x,y)|^2 &\leqslant \varphi(x,x) \varphi(y,y), \ \ \
\forall \, x, y \in D(\varphi). \label{2.2}
\end{align}
We put
\[
N_\varphi= \{ x \in D(\varphi) : \varphi(x,x)=0\}. \] By
\eqref{2.2}, we have
\[
N_\varphi= \{ x \in D(\varphi) : \varphi(x,y)=0, \ \ \ \forall \,
y \in D(\varphi) \},
\]
and so $N_\varphi$ is a subspace of $D(\varphi)$ and the quotient
space $D(\varphi) / N_\varphi := \{ \lambda_\varphi(x) \equiv
x + N_\varphi ; x \in D(\varphi) \}$ is a pre-Hilbert space with
respect to the inner product $$\ip{\lambda_\varphi(x)}{\lambda_\varphi(y)}
= \varphi(x, y), \quad x,y \in D(\varphi).$$ We
denote by $\H_\varphi$ the Hilbert space obtained by completion of $D(\varphi) / N_\varphi$.

\medskip
A positive sesquilinear form $\vp$  on $\A \times \A$ is said to be \emph{invariant}, and called an \emph{ips-form}, if
there exists a subspace $B(\varphi) $
of $\A$ (called a \emph{core} for $\varphi$) with the properties
\begin{itemize}
\item[({\sf ips}$_1$)] $B(\varphi) \subset R\A$;

\item[({\sf ips}$_2$)] $\lambda_\varphi(B(\varphi))$ is dense in $\H_\varphi$;

\item[({\sf ips}$_3$)]  $\varphi(xa, b) = \varphi(a, x\x b), \, \forall \, x\in \A, \forall \, a,b \in B(\varphi)$;

\item[({\sf ips}$_4$)] $\varphi(x\x a, yb) = \varphi(a, (xy)b), \, \forall \, x \in L(y), \forall \, a,b \in B(\varphi)$.
\end{itemize}
In other words, an ips-form is an \emph{everywhere defined} biweight, in the sense of \cite{ait_book}.

 To every ips-form $\vp$ on $\A$, with core $B(\varphi) $, there corresponds a triple $(\pi_\vp, \lambda_\vp, \H_\vp)$, where $\H_\vp$ is a Hilbert space,
$\lambda_\vp$ is a linear map from $B(\varphi) $ into $\H_\vp$ and $\pi_\vp$ is a *-representation of $\A$ in the
Hilbert space $\H_\vp$. We refer to  \cite{ait_book} for more details on this celebrated GNS construction.

Let $\A$ be a \pa\ and $\pi$ a *-representation of $\A$ in $\D(\pi)$. For $\xi \in \D(\pi)$ we put
\be \label{eq-phipixi}
 \vp_\pi^\xi (x,y) := \ip{\pi(x)\xi}{\pi(y)\xi}, \quad x,y \in \A.
\en
Then, $\vp_\pi^\xi$ is a positive sesquilinear form on $\A\times \A$.

Let $\BB\subseteq R\A$ and assume that $\pi (\BB)\subset \LDO{\D(\pi)}$. Then it is easily seen that $\vp_\pi^\xi$ satisfies the conditions ({\sf ips}$_3$) and ({\sf ips}$_4$) above. However, $\vp_\pi^\xi$ is not necessarily an ips-form since $\pi(\BB)\xi$ may fail to be dense in $\H$. For this reason, the following notion of \emph{regular} *-representation was introduced in \cite{ct_ban}.

\bedefi
A *-representation $\pi$ of $\A$ with domain $\D(\pi)$ is called $\BB$-\emph{regular} if $\vp_\pi^\xi$ is an ips-form with core $\BB$, for every $\xi \in \D(\pi)$.
\findefi

\berem The notion of regular *-representation was given in \cite{att_continuoushomom} for a larger class of positive sesquilinear forms (biweights) referring to the {\em natural} core
$$ B(\vp_\pi^\xi)=\{ a \in R\A: \pi(a) \in \D^{**}(\pi)\}$$ (we refer to \cite{ait_book} for precise definitions). If $\pi(\BB)\subset \LDO{\D(\pi)}$, the $\BB$-regularity implies that  $\vp_\pi^\xi$ is an also ips-form with core
$B(\vp_\pi^\xi)$. We will come back to this point in Proposition \ref{regrep}.
\enrem

\medskip
 Let $\A$ be a \pa. We assume that $\A$ is a locally convex Hausdorff vector space under the topology
$\tau$ defined by a (directed) set $\{p_\alpha\}_{\alpha \in \I}$ of seminorms. Assume that\footnote{\, Condition ({\sf cl}) was called
({\sf t1}) in \cite{antratsc}.}
\begin{itemize}

\item[({\sf cl})] for every $x \in \A$, the linear map $\OL_x: R(x)\mapsto \A$ with $\OL_x(y)=xy$, $y\in R(x)$,
is closed with respect to $\tau$, in the sense that, if $\{y_\alpha\}\subset R(x) $ is a net such that $y_\alpha \to y$   and $xy_\alpha \to z \in \A$, then $y\in R(x)$
and $z=xy$. 
 \end{itemize}
For short, we will say that, in this case, $\A$ is a \emph{topological partial *-algebra}. If the involution $x\mapsto x^*$ is continuous, we say that $\A$ is a \emph{*-topological} \pa.

 Starting from the family of seminorms $\{p_\alpha\}_{\alpha \in
\I}$, we can define a second topology $\tau\ha $ on $\A$ by
introducing the set of seminorms $\{ p\ha _\alpha(x)\}$, where
$$
p\ha _\alpha(x)= \max\{p_\alpha(x), p_\alpha(x\ha )\}, \quad x \in \A.
$$
The involution $x\mapsto x\ha $  is automatically
$\tau\ha $-continuous. By ({\sf cl}) it follows that, for every $x \in
\A$, both maps $\OL_x$, $\OR_x$ are $\tau\ha $-closed.  Hence, $\A[\tau\ha]$ is a *-topological \pa.

In this paper we will  consider the following particular classes of topological partial *-algebras.

\bedefi \label{def:multcore}
Let $\A[\tau]$ be a  topological partial *-algebra with locally convex topology $\tau$. Then,

\bee
\item
A subspace $\BB$ of $R\A$ is called a \emph{multiplication core} if
\begin{itemize}
\item[({\sf d}$_1$)]$e\in \BB$ if $\A$ has a unit $e$;
\item[({\sf d}$_2$)]$\BB\cdot \BB \subseteq \BB$;
\item[({\sf d}$_3$)] $\BB$ is $\tau\ha$-dense in $\A$;
\item[({\sf d}$_4$)] for every $b\in\BB$, the map $x\mapsto xb$ , $\, x\in \A$,  is $\tau$-continuous;
\item[({\sf d}$_5$)] one has $b\ha(xc) = (b\ha x)c, \; \forall \, x\in\A, b,c\in \BB$.
\end{itemize}

\item
 $\A[\tau]$ is called  \emph{$\Ao$-regular} if it possesses \footnote{\, In \cite{att_2010} it was only supposed that $\Ao$ is $\tau$-dense in $\A$.} a multiplication core $\Ao$ which is a *-algebra and, for every $b \in \Ao$, the map $x\mapsto bx$ , $x \in \A$, is $\tau$-continuous  (\cite[Def. 4.1]{att_2010}).
\ene
\findefi

\noi  If $\A$ is  \emph{$\Ao$-regular} and if, in addition, the  involution $x\mapsto x\ha$ is $\tau$-continuous for all $x\in\A$, then the couple
$(\A,\Ao)$ is   a locally convex \emph{quasi *-algebra}.

\berem \label{rem_semiass} A simple  limiting  argument shows that, if $\BB$ is an algebra (i.e., it is also associative),  then $\A$ is a $\BB$-right module, i.e.,
$$ (xa)b=x(ab), \; \forall\, x \in \A,\,a,b \in \BB.$$
If $\A$ is $\Ao$-regular then, in  a similar way,
$$ \mbox{$(xa)b=x(ab)$, $(ax)b = a(xb)$ for every $x \in \A$, $a,b \in \Ao$.}$$
\enrem

\berem We warn the reader that an $\Ao$-regular topological partial *-algebra $\A[\tau]$ is not necessarily a locally convex partial *-algebra in the sense of \cite[Def. 2.1.8]{ait_book}. Neither need it be  topologically regular in the sense of \cite[Def. 2.1.8]{att_2010}, which is a more restrictive notion.
\enrem

\berem Let $\A[\tau]$ be an $\Ao$-regular  topological partial *-algebra. Then, for every $b\in\Ao$, the maps $x\mapsto xb$ and $x\mapsto bx, \, x\in \A$, are also $\tau\ha$-continuous. However, the density of $\Ao$ in $\A[\tau\ha]$ may fail. Thus $\A[\tau\ha]$ need not be an $\Ao$-regular  *-topological partial *-algebra.
\enrem

\beexs
The three notions given in Definition \ref{def:multcore} are really different.

(1) Take $\A = \LDH$. Then, $R\A$ is given in  \eqref{eq:RLDH}, so that we have
 an example where  $R\A \cdot R\A \not \subset R\A$.

(2) Take again $\A = \LDH[{\sf t}_{s^\ast}]$. Then $\LDH$ is $\Ao$-regular for $\Ao= \LBD$.

(3) Assume $\LDH[{\sf t}_{s^\ast}]$   is self-adjoint, i.e. $\D  = \D\ha$ (for instance, when $\D=D^\infty(A)$ for a self-adjoint operator $A$). Then
$R\LDH = \{X\in \LBDH : X:\D \to\D\}$ is an algebra, but it is not *-invariant. Hence it is a multiplication core, since it is ${\sf t}_{s^\ast}$-dense in $\LDH$,  but  $\LDH$ is not $R\LDH$-regular.
\enexs

 The case   of a locally convex quasi *-algebra  $(\A,\Ao)$ was studied in \cite{FTT}
and a number of interesting properties have  been derived. Some of these extend to the general case of a \pa, as we shall see in the sequel.

\begin{prop}\label{regrep} Let $\A[\tau]$ be a topological partial *-algebra and $\BB$ a multiplication core. Then every $(\tau,{\sf t}_{s})$-continuous *-representation of $\A$ is $\BB$-regular.
\end{prop}
\begin{proof} First we may assume that $\pi(\BB)\subset \LDO{\D(\pi)}$. Indeed, put
\begin{align*}
& D(\pi_1):= \left\{\xi_0 + \sum_{i=1}^n \pi(b_i) \xi_i: \, b_i \in \BB,\, \xi_i \in \D(\pi); i=0, 1,\ldots, n \right\} , \\
& \pi_1(x)\left(\xi_0 +\sum_{i=1}^n \pi(b_i) \xi_i \right):= \pi(x)\xi_0  + \sum_{i=1}^n (\pi(x)\mult \pi(b_i))\xi_i.
\end{align*}
Then, exactly as in \cite{att_2010} we can prove that $\pi_1$ is a *-representation of $\A$ with $\pi_1(\BB)\subset \LDO{\D(\pi_1)}$.

If $\pi$ is $(\tau,{\sf t}_{s})$-continuous, then $\pi_1$ is $(\tau,{\sf t}_{s})$-continuous too (recall that domains are different!). Indeed, if $x_\alpha \stackrel{\tau}{\to} x$, then $\pi(x_\alpha)\xi \to \pi(x)\xi$, for every $\xi \in \D(\pi)$. The continuity of the right multiplication then implies that $x_\alpha b \stackrel{\tau}{\to} x b$, for every $b \in \BB$. Thus, by the continuity of $\pi$, we get, for every $b \in \BB$,  $\pi(x_\alpha b)\xi \to \pi(x b)\xi $, for every $\xi \in \D(\pi)$ or, equivalently, $(\pi(x_\alpha)\mult\pi(b))\xi \to (\pi(x)\mult\pi(b))\xi$, for every $\xi \in \D(\pi)$. Hence
\begin{align*}
\pi_1(x_\alpha) \left(\sum_{i=1}^n \pi(b_i) \xi_i \right)&= \sum_{i=1}^n (\pi(x_\alpha)\mult \pi(b_i))\xi_i \to\\
& \sum_{i=1}^n (\pi(x)\mult \pi(b_i))\xi_i = \pi_1(x)\left(\sum_{i=1}^n \pi(b_i) \xi_i\right).
\end{align*}
Thus, every $(\tau,{\sf t}_{s})$-continuous *-representation $\pi$ extends to a $(\tau,{\sf t}_{s})$-continuous \mbox{*-representation} $\pi_1$ with $\pi_1(\BB)\subset \LDO{\D(\pi_1)}$.
Finally we prove the  { $\BB$-regularity} of $\pi$. If $x \in \A$ then there exists a net $\{b_\alpha\} \subset \BB$ such that $b_\alpha \stackrel{\tau}{\to} x$. Then we have
$$
\| \lambda_{\vp_\pi^\xi}(x)-\lambda_{\vp_\pi^\xi}(b_\alpha)\|^2 = \vp_\pi^\xi(x-b_\alpha, x-b_\alpha)\leq p(x-b_\alpha)^2 \to 0,
$$
where $p$ is a convenient $\tau$-continuous seminorm. This implies that $\lambda_{\vp_\pi^\xi}(\BB)$ is dense in $\H_{\vp_\pi^\xi}$. Hence $\vp_\pi^\xi$ is an ips-form  {  with core $\BB$.}
\end{proof}

Let $\A[\tau]$ be a  topological partial *-algebra with multiplication core $\BB$ and $\vp$  a positive
 sesquilinear forms on $\A\times\A$ for which the conditions ({\sf ips}$_1$), ({\sf ips}$_3$) and ({\sf ips}$_4$) are satisfied (with respect to $\BB$). Suppose that
$\varphi$ is $\tau$-continuous, i.e., there exist $p_\alpha$, $\gamma>0$ such that:
$$
|\varphi(x,y)|\leq \gamma \,p_\alpha(x)p_\alpha(y) \quad \forall x,y \in \A.
$$
Then ({\sf ips}$_2$) is also satisfied and, therefore,
$\BB$ is a core for $\varphi$, so that $\varphi$ is an ips-form.
We denote by $\pppb$   the set of \emph{all} $\tau$-continuous ips-forms with core $\BB$.

Using the continuity of the multiplication and Remark \ref{rem_semiass}, it is easily seen that
if $\vp \in \pppb$ and $a\in \BB$, then $\vp_a \in \pppb$, where
$$ \vp_a(x,y):= \vp(xa, ya), \quad x,y \in \A.$$

\section{Topological partial *-algebras with sufficiently many *-representations}
\label{sect-suffmanyrep}

Throughout this paper we will be mostly concerned with topological partial *-algebras possessing sufficiently many continuous *-representations.
In the case of topological *-algebras this situation can be studied by introducing the so-called (topological) *-radical of the algebra. Thus we extend this notion to topological \pa s.

\medskip Let $\A[\tau]$ be a topological partial *-algebra. We define the \emph{*-radical} of $\A$ as
$$
\R\ha(\A):=\{x\in\A:\, \pi(x)=0,\, \mbox{for all}\, (\tau,{\sf t}_{s})\mbox{-continuous *-representations}\,\, \pi\}.
$$
We put $\R^*(\A)=\A$, if $\A[\tau]$ has no $(\tau,{\sf t}_{s})\mbox{-continuous *-representations}$.

\berem
The *-radical was defined in \cite[Sec.5]{att_2010} as
$$
\R\ha_{*}(\A):=\{x\in\A:\, \pi(x)=0,\, \mbox{for all}\, (\tau,{\sf t}_{s*})\mbox{-continuous *-representations}\,\, \pi\}.
$$
However, the two definitions are equivalent. Indeed, since every  $(\tau,{\sf t}_{s*})$-continuous *-representation is  $(\tau,{\sf t}_{s})$-continuous, we have $\R\ha_{*}(\A) \subset \R\ha(\A)$. In order to  prove that $\R\ha(\A) \subset \R_{*}\ha(\A)$, assume that
$x\not\in\R\ha_{*}(\A) $, i.e., there is a  $(\tau,{\sf t}_{s*})$-continuous *-representation $\pi$ such that $\pi(x)\neq 0$.
But $\pi$  is also  $(\tau,{\sf t}_{s})$-continuous, hence  $x\not\in\R\ha(\A) $ as well.
\enrem

The  *-radical enjoys the following  { immediate} properties:

\begin{itemize}
\item[(1)] If $x \in \R^*(\A)$, then $x\ha \in \R^*(\A)$.
\item[(2)] If $x\in \A$, $y \in \R^*(\A)$ and $x\in L(y)$, then $xy \in \R^*(\A)$.
\end{itemize}

From now on, we denote by $\rep$  {  the set} of all $(\tau, {\sf t}_s)$-continuous  *-represen\-ta\-tions  of $\A$.
It $\A$ has a multiplication core $\BB$,  we may always suppose that $\pi(x) \in \LDO{\D(\pi)}$ for every $x \in \BB$,  {  as results from the proof of Proposition \ref{regrep}}.
\begin{prop}\label{prop_finalnew}
Let $\A[\tau]$ be a topological partial *-algebra with unit $e$. Let $\BB$ be a  multiplication core. For an element $x\in \A$ the following statements are equivalent.
\begin{itemize}
  \item[(i)] $x \in \R^*(\A)$.
  \item[(ii)]  $\vp(x,x)=0$  for every $\vp \in \pppb$.
\end{itemize}
\end{prop}

\begin{proof}(i) $\Rightarrow $ (ii): Let $\vp \in \pppb$ and $\pi_\vp$ the corresponding GNS representation. Then, for every $x \in \A$,
$$\|\pi_\vp(x)\lambda_\vp(a)\|^2 = \vp(xa,xa) =\vp_a(x,x) \leq p(x)^2, \quad a \in \BB$$
for some continuous $\tau$-seminorm $p$ (depending on $a$). Hence $\pi_\vp$ is $(\tau,{\sf t}_{s})$-continuous. If $x\in \R^*(\A)$, then $\pi_\vp(x)=0$. Thus $\vp(xa,xa)=0$, for every $a \in \BB$. From $e \in \BB$, we get the statement.

(ii) $\Rightarrow$ (i) Let $\pi\in \rep$. We assume $\pi(\BB) \subset \LDO{\D(\pi)}$.
For $x,y \in \A$ and $\xi \in \D$, put, as before,
$$ \vp_\pi^\xi(x,y):=\ip{\pi(x)\xi}{\pi(y)\xi}, \quad x,y \in \A.$$
Then,
$$ |\vp_\pi^\xi(x,y)|=|\ip{\pi(x)\xi}{\pi(y)\xi}| \leq \|\pi(x)\xi\|\|\pi(y)\xi\|\leq p(x)p(y)$$
for some $\tau$-continuous seminorm $p$. Hence, $\vp_\pi^\xi$ is continuous.

Thus, $\vp_\pi^\xi \in \pppb$ and, by the assumption, $\|\pi(x)\xi\|^2=\vp_\pi^\xi(x,x)=0$. The arbitrariness of $\xi$ implies that $\pi(x)=0$.
\end{proof}

As for topological *-algebras, the *-radical contains all elements $x$ whose \emph{square} $x\ha x$ (if well defined) vanishes.

\begin{prop} Let $\A$ be a topological \pa. Let $x \in \A$, with $x\ha \in L(x)$. If
$x\ha x=0$, then $x \in \R^*(\A)$
\end{prop}
\begin{proof} If $\pi$ is a $(\tau,{\sf t}_{s})$-continuous *-representation of $\A$, $\pi(x\ha)\mult\pi(x)=\pi(x)\ad\mult\pi(x)$ is well-defined and equals $0$. Hence, for every $\xi \in \D(\pi)$,
\begin{align*} \|\pi(x)\xi\|^2 &= \ip{\pi(x)\xi}{\pi(x)\xi}\\
&=\ip{\pi(x)\ad\mult\pi(x) \xi}{\xi}= \ip{\pi(x\ha)\mult\pi(x) \xi}{\xi}\\ &=\ip{\pi(x\ha x)\xi}{\xi} =0.
\end{align*}
Hence $\pi(x) = 0$.
\end{proof}

\berem   A sort of converse of the previous statement was stated in \cite[Proposition 5.3]{att_2010}. Unfortunately, the proof given there contains a gap.
\enrem

\bedefi A topological partial *-algebra $\A[\tau]$ is called \emph{*-semisimple} if, for every $x\in \A\setminus\{0\}$ there exists a $(\tau, {\sf t}_s)$-continuous
*-representation $\pi$ of $\A$ such that $\pi(x)\neq 0$ or, equivalently, if $\R^*(\A)=\{ 0\}$.
\findefi

By Propositition \ref{prop_finalnew}, $\A[\tau]$ is *-semisimple if, and only if, for some multiplication core $\BB$,  the family  of ips-forms  $\pppb$ is \emph{sufficient}
in the following sense \cite{att_2010}.
\bedefi
A   family $\mathcal M$ of continuous ips-forms on $\A\times\A$ is {\it sufficient}
 if $x\in \A$ and $\vp(x,x)=0$ for every $\vp \in {\mc M}$ imply $x=0$.
\findefi

The sufficiency of the family $\M$ can be described in several different ways.

\begin{lemma} \label{lem:Msuff}
Let $\A$ be a topological \pa\ with multiplication core $\BB$. Then the following statements are equivalent:
\begin{itemize}
\item[(i)] $\M$ is sufficient.
\item[(ii)]$\vp(xa,b)=0$, for every $\vp \in \M$ and $a,b \in \BB$, implies $x=0$.
\item[(iii)]$\vp(xa,a)=0$, for every $\vp \in \M$ and $a \in \BB$, implies $x=0$.
\item[(iv)]$\vp(xa,y)=0$, for every $\vp \in \M$ and $y \in \A$, $a\in \BB$, implies $x=0$.
\item[(v)]$\vp(xa,xa)=0$ for every $\vp \in \M$ and $a\in \BB$, implies $x=0$.
\end{itemize}
\end{lemma}
We omit the easy proof.

\medskip Of course, if the family $\M$ is sufficient, any larger family $\M' \supset \M$ is also sufficient. In this case, the maximal sufficient family (having $\BB$ as core) is obviously
the set $\pppb$ of \emph{all} continuous ips-forms with core $\BB$. Hence if a sufficient family $\M \subseteq \pppb$ exists, $\A[\tau]$ is *-semisimple.

\beex \label{ex_33}As mentioned before, the space $\LDH$ is a $\LBD$-regular \pa, when endowed with the strong*-topology ${\sf t}_{s^*}$.
The set of positive sesquilinear forms $\M:= \{ \vp_\xi: \xi \in \D\}$, where $\vp_\xi(X,Y)=\ip{X\xi}{Y\xi}$, $X,Y \in \LDH$, is a sufficient family of
ips-forms with core $\LBD$. Indeed, if $\vp_\xi(X,X)=0$, for every $\xi \in \D$, then $\|X\xi\|^2=0$ and therefore $X=0$.
\enex

  \beex \label{ex_34} As shown in \cite{bt_ellepi}, the space $L^p(X)$, $X=[0,1]$, endowed with its usual norm topology, is $L^\infty(X)$-regular and it is *-semisimple if  $p\geq 2$. Indeed, in this case the family of all continuous ips-forms is given by
$\M=\{ \vp_w :  w \in L^{p/(p-2)}, w \geq 0\}$, where
$$\vp_w(f,g)=\int_X f(t)\overline{g(t)}w(t)dt, \quad f,g \in L^p(X),$$ and it is sufficient.

 If $1\leq p<2$,
the set of all continuous ips-forms reduces to $\{0\}$. Hence, in this case, $\R^*(L^p(X))=L^p(X)$.
\enex

 Let $\A$ be a topological \pa\ with multiplication core $\BB$. If $\A$ possesses a sufficient family  $\mathcal M$ of ips-forms,
an \emph{extension} of the multiplication of $\A$ can be introduced in a way similar to  \cite[Sec.4]{att_2010}.

We say that the \emph{weak} multiplication $x\wmult y$ is well-defined ( {  with respect}  to $\M$) if there exists $z\in\A$ such that:
$$
\varphi(ya,x^*b)=\varphi(za,b),\;\forall\, a,b\in\BB, \forall\,\varphi\in\mathcal M.
$$
In this case, we put $x\wmult y:=z$ and the sufficiency of ${\mathcal M}$ guarantees that $z$ is unique.
  The weak multiplication $\wmult$ clearly depends on  $\mathcal M$: the larger is $\M$, the stronger is the weak multiplication, in the sense that if $\M \subseteq \M' \subseteq \pppb$ and $x\wmult y$ exists w.r. to $\M'$, then $x\wmult y$ exists with respect to $\M$ too.

A handy criterion for the existence of the weak multiplication is provided by the following
\begin{prop}\label{prop_43} Let $\BB$ be an algebra, then the weak product $x\wmult y$ is defined (with respect to $\M$) if, and only if, there exists a net $\{b_\alpha\}$ in $\BB$ such that $b_\alpha\stackrel{\tau}\longrightarrow y$
and $x b_\alpha\stackrel{\tau^{\scriptscriptstyle\M} _w}\longrightarrow z\in\A.$
\end{prop}
 Here $\tau^{\scriptscriptstyle\M} _w$ is the weak topology determined by $\mathcal M$, with seminorms
$x\mapsto |\vp(xa,b)|, $ $\vp \in {\mathcal M}, a,b \in \BB$.
It is easy to prove that $\A$  is also a  partial *-algebra with respect to the weak multiplication.

Since it holds in typical examples,  e.g. $\LDH$ \cite[Prop. 3.2]{att_2010},  we will often suppose that the following condition is satisfied:
\begin{itemize}
\item[{\sf (wp)}] $xy$ exists if, and only if, $x\wmult y$ exists. In this case $xy=x\wmult y$.
\end{itemize}

In this situation it is possible to define a stronger multiplication on $\A$:
we say that the {\it strong} multiplication $x\bullet y$
is well-defined (and that $x\in L^s(y)$ or $y\in R^s(x)$) if $x\in L(y)$ and:
\begin{itemize}
\item[({\sf sm}$_1$)]$\varphi((xy)a,z^*b)=\varphi(ya,(x^*z^*)b), \; \forall \,z \in L(x),\forall\,\varphi\in{\M}, \forall\, a,b\in \BB$;
\item[({\sf sm}$_2$)]$\varphi((y^*x^*)a,vb)=\varphi(x^*a,(yv)b), \; \forall\, v \in R(y),\forall\, \varphi\in{\M},\forall\, a,b\in\BB.$
\end{itemize}
The same considerations on the dependence on $\M$ of the weak multiplication apply, of course, to the strong multiplication.

\bedefi Let $\A$ be a partial *-algebra. A *-representation $\pi$ of $\A$ is called  \emph{quasi-symmetric} if, for every $x \in \A$,
\begin{align*}
& \bigcap_{z \in L(x)} D((\pi(x)^*\upharpoonright \pi(z^*)\D(\pi))^*= D(\overline {\pi(x)}) ;\\
&  \bigcap_{v \in R(x)} D((\pi(x^*)^*\upharpoonright \pi(v)\D(\pi))^* =  D(\overline {\pi(x)\ad}).
\end{align*}
\findefi
Of course, the same definition can be given for any partial O*-algebra $\MM$ (by considering the identical *-representation).
\berem The conditions given in the previous Definition are certainly satisfied if, for every $x \in \A$, there exist $s \in L(x), \,t \in R(x)$ such that $(\pi(x)^*\upharpoonright \pi(s^*)\D(\pi))^*= \overline{\pi(x)}$ and $(\pi(x^*)^*\upharpoonright \pi(t)\D(\pi))^* = \overline{\pi(x^*)}$.
These stronger conditions are  {  satisfied}, in particular,  by any
symmetric O*-algebra $\MM$ ({\em symmetric} means that $(I+X^*\overline{X})^{-1}$ is in the bounded part of $\MM$, for every $X \in \MM$). The proof given in \cite[Theorem 3.5]{att_2010}, shows that in this case $\MM$ is quasi-symmetric, whence the name. In particular, $\LDH$ is quasi-symmetric.\enrem

\begin{prop}Let $\A$ be a *-semisimple topological \pa\ with  multiplication core $\BB$ containing the unit $e$ of $\A$. If the strong product $x\bullet y$ of  $x,y\in\A$ is
 well-defined with respect to $\pppb$, then for every quasi-symmetric $\pi\in\repb$ with $\pi (\BB) \subset \LDO{\D(\pi)}$ and $\pi(e)=I_{\D(\pi)}$, the strong product $\pi(x)\circ\pi(y)$ is well-defined too and
 $$
  \pi(x\bullet y)=\pi(x)\circ\pi(y).
  $$
\end{prop}
\begin{proof} Let $\pi\in\repb$ satisfy the required assumptions. Let $\xi \in \D(\pi)$ and define $\vp_\pi^\xi$ as in the proof of Proposition \ref{prop_finalnew}. Then $\vp_\pi^\xi \in \pppb$ and since $x\bullet y$ is well-defined (with respect to $\pppb$), we have
\begin{align*}
&\vp_\pi^\xi((xy)a,z^*b)=\vp_\pi^\xi(ya,(x^*z^*)b), \; \forall \,z \in L(x), \forall\, a,b\in \BB;\\
&\vp_\pi^\xi((y^*x^*)a,vb)=\vp_\pi^\xi(x^*a,(yv)b), \; \forall\, v \in R(y),\forall\, a,b\in\BB.
\end{align*}
 {  Equivalently,}
\begin{align*}
& \ip{(\pi(x)\mult \pi(y))\pi(a)\xi}{\pi(z^*)\pi(b)\xi}=\ip{\pi(y)\pi(a)\xi}{(\pi(x^*)\mult\pi(z^*))\pi(b)\xi},\\ &\hspace{5cm} \forall \,z \in L(x), \forall\, a,b\in \BB;\\
& \ip{(\pi(y^*)\mult\pi(x^*))\pi(a)\xi}{\pi(v)\pi(b)\xi}=\ip{\pi(x^*)\pi(a)\xi}{(\pi(y)\mult \pi(v))\pi(b)\xi},\\ &\hspace{5cm}\forall\, v \in R(y),\forall\, a,b\in\BB.
\end{align*}
By taking $a=b=e$ and using the polarization identity, one gets, for every $\xi, \eta \in \D(\pi)$,
\begin{align*}
& \ip{(\pi(x)\mult \pi(y))\xi}{\pi(z^*)\eta}=\ip{\pi(y)\xi}{(\pi(x^*)\mult\pi(z^*))\eta}, \;\forall \,z \in L(x) ;\\
& \ip{(\pi(y^*)\mult\pi(x^*))\xi}{\pi(v)\eta}=\ip{\pi(x^*)\xi}{(\pi(y)\mult \pi(v))\eta},\;\forall\, v \in R(y).
\end{align*}
From  {  these relations,} it follows that
\begin{align*}
&\pi(y): \D(\pi)\to D((\pi(x)^*\upharpoonright \pi(z^*)\D(\pi))^*,\; \forall z\in L(x), \\
&\pi(x^*): \D(\pi)\to D((\pi(y^*)^*\upharpoonright \pi(v)\D(\pi))^*,\; \forall v\in R(y).
\end{align*}
By the assumption, it follows that $\pi(y):\D(\pi) \to D(\overline{\pi(x)})$ and $ \pi(x^*):\D(\pi) \to D(\overline{\pi(y^*)})$. Thus, $\pi(x)\circ \pi(y)$ is well defined.
\end{proof}

\section{ Representable functionals versus ips-forms}
\label{sect-represfunct}

So far  {  we have used} ips-forms in order to characterize *-semisimplicity of a topological \pa. The reason lies in the fact that ips-forms allow a GNS-like construction. However, from a general point
of view it is not easy to find conditions for the existence of sufficient families of ips-forms,  {  whereas} there exist well-known criteria for the existence of continuous \emph{linear} functionals that separate points of $\A$.
However continuous linear functionals, which are positive in a certain sense, do not give rise, in general to a GNS construction.
This can be done, if they are \emph{representable} \cite{bit_reprfunct} in the sense specified below. It is then natural to consider, in more details, conditions for the representability of continuous positive
 linear functionals.

\bedefi Let $\A[\tau]$ be a topological \pa\ with multiplication core $\BB$. A continuous linear functional $\omega$ on $\A$ is \emph{$\BB$-positive }if $\omega (a^*a)\geq 0$ for every $a \in \BB$.
\findefi
The continuity of $\omega$ implies that $\omega (x)\geq 0$ for every $x$ which belongs to the $\tau$-closure $\A^+(\BB)$ of the set
$$
\BB^{(2)}=\left\{\sum_{k=1}^n x_k^* x_k, \, x_k \in \BB,\, n \in {\mb N}\right\}.
$$

In the very same way as in \cite[Theorem 3.2]{FTT} one can prove the following
\begin{thm} \label{thm_separation}
 Assume that $\A^+(\BB) \cap (-\A^+(\BB))=\{0\}$. Let $a \in \A^+(\BB)$, $a \neq 0$. Then there exists a
continuous linear functional $\omega$ on $\A$ with the properties:
\begin{itemize}
  \item[(i)] $\omega (x)\geq 0$, $\; \forall \ x \in \A^+(\BB)$;
  \item[(ii)] $\omega(a)>0$.
\end{itemize}
\end{thm}
\noi {  The set $\A^+(\BB)$ will play an important role  in Theorem \ref{thm_frimplies semis}  and in the analysis of order bounded elements in Section \ref{subsec-orderbdd}.}

\bedefi \label{defn_representable}
Let $\omega$ be a linear functional on $\A$ and $\BB$ a subspace of $R\A$. We say that $\omega$ is  \emph{representable} (with respect to $\BB$) if  the following requirements are satisfied:

({\sf r}$_1$) $\omega(a^*a)\geq 0$ for all $a\in\BB$;

({\sf r}$_2$) $\omega(b^*(x^*a))=\overline{\omega(a^*(xb))},\;\forall\,\,a,b\in\BB$, $x\in\A$;

({\sf r}$_3$) $\forall x\in\A$ there exists $\gamma_x>0$ such that
$|\omega(x^*a)|\leq \gamma_x\,\omega(a^*a)^{1/2}$, for all $a \in \BB$.

\findefi
\noi In this case, one can prove that there exists a triple $(\pi^\BB_{\omega}, \lambda^\BB_{\omega}, \H^\BB_{\omega})$ such that
\begin{itemize}
  \item[(a)] $\pi^\BB_{\omega}$ is  a *-representation of $\A$ in $\H_\omega$;

  \item[(b)] $\lambda^\BB_{\omega}$ is a linear map of $\A$ into
  $\H^\BB_{\omega}$ with $\lambda^\BB_{\omega}(\BB)=\D({\pi^\BB_\omega})$ and
   $\pi^\BB_{\omega}(x)\lambda^\BB_{\omega}(a)=\lambda^\BB_{\omega}(xa)$, for every $x \in\A,\,a \in \BB$.

  \item[(c)] $\omega(b^*(xa))=\ip{\pi^\BB_{\omega}(x)\lambda^\BB_\omega(a)}{\lambda^\BB_\omega(b)}$,  for every $x \in \A$, $a,b \in \BB$.

\end{itemize}
In particular, if $\A$ has a unit $e$ and $e \in \BB$, we have:
\begin{itemize}
  \item[(a$_1$)] $\pi^\BB_{\omega}$ is a cyclic *-representation of $\A$ with cyclic vector $\xi_\omega$;

  \item[(b$_1$)] $\lambda^\BB_{\omega}$ is a linear map of $\A$ into
  $\H^\BB_{\omega}$ with $\lambda^\BB_{\omega}(\BB)=\D({\pi^\BB_\omega})$, $\xi_\omega=\lambda^\BB_{\omega}(e)$ and
    $\pi^\BB_{\omega}(x)\lambda^\BB_{\omega}(a)=\lambda^\BB_{\omega}(xa)$, for every $x \in\A,\, a\in \BB$.

  \item[(c$_1$)] $\omega(a)=\ip{\pi^\BB_{\omega}(x)\xi_\omega}{\xi_\omega}$, for every $x \in \A$.
\end{itemize}

The GNS construction then depends on the subspace $\BB$. We adopt the notation ${\mc R}(\A,\BB)$ for denoting the set of linear functionals on $\A$ which are representable with respect to the same $\BB$.
\berem
 It is worth recalling (also for fixing notations) that the Hilbert space $\H_\omega$ is defined by considering the subspace of $\BB$
    $$
    N_\omega =\{x \in \BB;\, \omega(y^*x)=0, \;\forall\, y \in \BB\}.
    $$
 The quotient  $\BB/N_\omega\equiv\{\lambda_\omega^0(x):=x+N_\omega; x \in \BB\}$ is a pre-Hilbert space with inner product
    $$
     \ip{\lambda_\omega^0(x)}{\lambda_\omega^0(y)}= \omega(y^*x), \quad x,y \in \BB.
    $$
 Then $\H_\omega$ is the completion of
    $\lambda_\omega^0(\BB)$. The representability of $\omega$ implies that $\lambda_\omega^0: \BB\to  \H_\omega$ extends to a linear map $\lambda_\omega:\A \to\H_\omega$.
\enrem

\berem  We notice that if $\pi$ is a *-representation of $\A$ on the domain $\D(\pi)$, and $\BB$ is a subspace of $R\A$ such that $\pi(\BB) \subset \LDO{\D(\pi)}$, then, for every $\xi \in \D(\pi)$,
the linear functional
$\omega_\pi^\xi$ defined by $\omega_\pi^\xi (x)= \ip{\pi(x)\xi}{\xi}$ is representable, whereas the corresponding sesquilinear form $\vp_\pi^\xi (x,y)= \ip{\pi(x)\xi}{\pi(y)\xi}$
is not necessarily an ips-form; the latter fact leads to the notion of regular representation discussed above.
\enrem

\beex \label{ex_contnonrepr}A continuous linear functional $\omega$ whose restriction to $\BB$ is positive need not be representable. As an example, consider $\A=L^1(I)$, $I$ a bounded interval
on the real line, and $\BB= L^\infty (I)$. The linear functional
$$ \omega(f)= \int_I f(t)dt , \quad f \in L^1(I)$$ is continuous, but it is not representable, since ({\sf r}$_3$) fails if $f \in L^1(I)\setminus L^2(I)$.
\enex

Since  multiplication cores  play  {  an important role} for   topological \pa s, we restrict our attention to the case where $\BB$ is a multiplication core
and we omit explicit reference to $\BB$ whenever it appears.
We will denote by  ${\mc R}_c(\A,\BB)$ the set of $\tau$-continuous linear functionals that are representable (with respect to $\BB$).

Since both representable functionals and ips-forms define GNS-like representations it is natural to consider the interplay of these two notions, with particular reference to the topological case.
We refer also to \cite{bit_reprfunct, FTT} for more details.

\begin{prop}\label{3.3} Let $\A$ be a topological \pa\ with  multiplication core $\BB$ which is an algebra.
If $\vp$ is a $\tau$-continuous ips-form on $\A$ then, for every $b \in \BB$, the linear functional $\omega_\vp^b$ defined by
$$ \omega_\vp^b(x)= \vp(xb,b), \quad x\in \A$$
is representable and the corresponding map $a \in \BB \mapsto \lambda_{\omega_\vp^b}^0(a)\in \H_{\omega_\vp^b}$ is continuous.

Conversely, assume that $\A$ has a unit $e\in \BB$. Then, if $\omega$ is a representable linear functional on $\A$ and  the map $a \in \BB \mapsto \lambda_\omega^0(a)\in \H_\omega$ is continuous,
the  positive sesquilinear form $\vp_\omega$ defined on $\BB\times \BB$ by
$$ \vp_\omega(a,b):= \omega (b^*a),$$ is $\tau$-continuous on $\BB\times \BB$ and it extends to a continuous ips-form $\widetilde{\vp}_\omega$ on $\A$.
\end{prop}
\begin{proof} We prove only the second part of the statement.

For every $a,b\in \BB$ we have
\begin{eqnarray*}
 |\vp_\omega (a,b)| &=& |\omega (b^*a)|=|\ip{\pi_\omega(a)\lambda_\omega^0(e)}{\pi_\omega(b)\lambda_\omega^0(e)}|\\
&\leq& \|\pi_\omega(a)\lambda_\omega^0(e)\| \|\pi_\omega(b)\lambda_\omega^0(e) \| = \|\lambda_\omega^0(a)\| \|\lambda_\omega^0(b)\|\leq p(a)p(b)
 \end{eqnarray*}
for some continuous seminorm $p$.

Hence $\vp_\omega$ extends uniquely to $\A \times \A$. Let $\widetilde{\vp}_\omega$ denote this extension. It is easily seen that $\widetilde{\vp}_\omega$ is a positive sesquilinear form on $\A \times \A$ and
$$
 |\widetilde{\vp}_\omega(x,y)| \leq p(x) p(y), \quad \forall x,y \in \A.
$$
Hence the map $x \mapsto \lambda_{\widetilde{\vp}_\omega}(x)\in \H_{\widetilde{\vp}_\omega}$ is also continuous, since
$$
 \|\lambda_{\widetilde{\vp}_\omega}(x)\|^2= \widetilde{\vp}_\omega(x,x)\leq p(x)^2, \forall x \in \A.
 $$
Thus, if $x=\tau-\lim_\alpha b_\alpha$, $b_\alpha \in \BB$, we get
$$
\|\lambda_{\widetilde{\vp}_\omega}(x) - \lambda_{\widetilde{\vp}_\omega}(b_\alpha)\|^2= \widetilde{\vp}_\omega(x-b_\alpha, x-b_\alpha)\leq p(x-b_\alpha)^2\to 0.
$$
The conditions ({\sf ips}$_3$) and ({\sf ips}$_4$) are readily checked.   {  Concerning ({\sf ips}$_4$), for instance, } let $x \in L(y)$ and $a,b \in \BB$. Then,
\begin{align*}
 \widetilde{\vp}_\omega(a,(xy)b) &= \omega(((x y)b)\ha a)=\omega(b\ha(xy)\ha a) \\
   &=\ip{\pi_\omega(xy)\ad \lambda_\omega^0(a)}{\lambda_\omega^0(b)}\\
   &= \ip{(\pi_\omega(y)\ad \mult\pi_\omega(x)\ad) \lambda_\omega^0(a)}{\lambda_\omega^0(b)} \\
   &= \ip{ \pi_\omega(x)\ad \lambda_\omega^0(a)}{\pi_\omega(y)\lambda_\omega^0(b)}\\
   &=\widetilde{\vp}_\omega(x^*a,yb).
\end{align*}
\vspace*{-12mm}

\end{proof}

\berem If $\omega$ is a representable linear functional on $\A$ and the map $a \in \BB \mapsto \lambda_\omega^0(a)\in \H_\omega$ is continuous, then $\omega$ is continuous. The converse is false in general.

\enrem

 However, the continuity of $\omega$  implies  the $\tau^*$-closability of the map $\lambda_\omega^0: a \in \BB \mapsto \lambda_\omega^0(a)\in \H_\omega$ as the next proposition shows.
\begin{prop}
\label{prop:closable}
Let $\omega$ be continuous and $\BB$-positive. Then the map $\lambda_\omega^0: a \in \BB \mapsto \lambda_\omega^0(a)\in \H_\omega$ is $\tau^*$-closable.
\end{prop}

\begin{proof} Let $a_\delta\stackrel{\tau^*}{\to} 0$, $a_\delta \in \BB$, and suppose that the net  $\{\lambda_\omega^0(a_\delta)\}$
is Cauchy in $\H_\omega$.
 Hence it converges to some $\xi \in
\Hil_\omega$ and
$$
 \ip{\lambda_\omega^0(b)}{\lambda_\omega^0(a_\delta)} \to \ip{\lambda_\omega^0(b)}{\xi}, \quad \forall \ b \in \BB.
 $$
  Moreover,
$$
\ip{\lambda_\omega^0(b)}{\lambda_\omega^0(a_\delta)} = \omega(a_\delta^* b)\to 0, \quad \forall \ b \in \BB,
$$
since $a_\delta^*\stackrel{\tau}{\to}{0}$ and  the right multiplication by $b \in \BB$ and $\omega$ are   { both} $\tau$-continuous.
Thus $\ip{\lambda_\omega^0(b)}{\xi}=0$, for every $b \in \BB$. This implies that $\xi =0$  { and, therefore, } $\lambda_\omega^0(a_\delta)\to 0.$
\end{proof}

Actually, it is easy to see that the closability of the map $\lambda_\omega^0$ is equivalent to the closability of $\vp_\omega$. Indeed, closability of the map $a \in \BB \mapsto \lambda_\omega^0(a)\in \H_\omega$
means that if $a_\delta\stackrel{\tau^*}{\to} 0$ and $\{\lambda_\omega^0 (a_\delta)\}$ is a Cauchy net, then $\lambda_\omega^0 (a_\delta)\to 0$. But $\{\lambda_\omega^0 (a_\delta)\}$
 is a Cauchy net if and only if $\vp_\omega(a_\delta -a_\gamma,a_\delta -a_\gamma)\to 0$. This leads to the conclusion $\|\lambda_\omega^0 (a_\delta)\|^2=\vp_\omega(a_\delta,a_\delta)\to 0$.

Therefore, Proposition \ref{prop:closable}  generalizes \cite[Prop. 2.7]{FTT},  which  says  {  that, for a locally convex quasi *-algebra $(\A, \Ao)$,  the sesquilinear} form $\vp_\omega$
is closable if $\omega \in {\mc R}_c(\A,\Ao)$.

Thus, if $\omega$ is continuous and $\BB$-positive, the map $\lambda_\omega^0$ has a closure $
\overline{\lambda_\omega^0}$ defined on
$$
D(\overline{\lambda_\omega^0})=\{ x \in \A: \exists \{a_\delta\}\subset \BB, a_\delta \stackrel{\tau^*}{\to}x, \, \mbox{$\{\lambda_\omega^0 (a_\delta)\}$ is a Cauchy net}\}.
$$
From the  discussion above, it follows that $D(\overline{\lambda_\omega^0})$ coincides with the domain $D(\overline{\vp_\omega})$ of the closure of
$\vp_\omega$.

\medskip  For the case of a locally convex quasi *-algebra $(\A,\Ao)$, the following assumption was made in \cite{FTT}:
\begin{itemize}
\item[(\sf fr)] $\displaystyle \bigcap_{\omega \in {\mc R}_c(\A,\BB)} D(\overline{\vp_\omega})=\A$.
\end{itemize}
Quasi \mbox{*-algebras} verifying the condition {(\sf fr)} are called \emph{fully representable} (hence the acronym).  Some concrete examples   have been described in \cite{FTT}
and several interesting structure properties have been derived. We maintain the same definition and the same name in the case of  {  topological \pa s
and, in complete analogy,} we say that a topological \pa\ $\A[\tau]$, with multiplication core $\BB$ is {\em fully representable} if
\begin{itemize}
\item[(\sf fr)] $D(\overline{\vp_\omega})=\A$, for every  $\omega\in \rcab$.
\end{itemize}

Then we have the following generalization of Proposition 3.6 of \cite{FTT}.
\begin{prop}\label{prop: 3.10} Let $\A$ be a  { semi-associative} *-topological \pa\ with  multiplication core $\BB$. Assume that $\A$ is fully representable and let $\omega\in \rcab$.
Then,    $\overline{\vp_\omega}$ is an ips-form on $\A$ with core $\BB$, with the property
$$ \overline{\vp_\omega}(xa,b)=\omega (b^*xa), \quad \forall x \in \A,\, a,b \in \BB.$$

\end{prop}
\begin{proof}
The continuity of the involution implies that, for every $a\in \BB$, the map $x\mapsto a\ha x$ is continuous on $\A$. Hence the linear functional $\omega_a$ defined by $\omega(a\ha x a)$ is continuous.
We now prove that $\omega_a$ is representable; for this  we need to check properties ({\sf r}$_1$), ({\sf r}$_2$) and ({\sf r}$_3$). We have
$$
\omega_a (b\ha b)= \omega (a\ha(b\ha b) a) = \omega ((a\ha b\ha)(ba) )\geq 0, \quad \forall b \in \BB,
$$
i.e., ({\sf r}$_1$) holds.
Furthermore, for every $b,c\in \BB$, we have
\begin{align*}
 \omega_a(c\ha (xb))&=\omega(a\ha (c\ha (xb)) a)= \omega (a\ha ((c\ha x)b)) a)\\&= \overline{\omega (a\ha (b\ha (x\ha c)) a)}=\overline{\omega_a(b\ha (x\ha c))}.
\end{align*}
As for ({\sf r}$_3$), for every $x \in \A$ and $b\in \BB$, we have
\begin{align*}
 |\omega_a(x\ha b)|&=|\omega(a\ha(x\ha b)a)|=|\omega((a\ha x\ha) (ba))| \\ &\leq \gamma_{x,a}\omega (a\ha(b\ha b)a)^{1/2}=\gamma_{x,a}\omega_a(b\ha b)^{1/2}.
\end{align*}
Thus $\omega_a \in \rcab$, for every $a \in \BB$. By Proposition \ref{prop:closable}, $\overline{\vp_{\omega_a}}$ is well-defined and, by the assumption, $D(\overline{\vp_{\omega_a}})=\A$.
 Hence, if $x \in \A$, there exists a net $\{x_\alpha\}\subset \BB$ such that $x_\alpha \stackrel{\tau\ha} \to x$ and
{  $\vp_{\omega_a}(x_\alpha - x_\beta, x_\alpha - x_\beta) \to 0$ or,
equivalently, \mbox{$\vp_\omega ((x_\alpha - x_\beta)a, (x_\alpha - x_\beta)a)$} $ \to 0$}. Hence, by the definition of closure, for every $b \in \A$,
$$
\overline{\vp_\omega} (xa,b)=\lim_\alpha \vp_\omega(x_\alpha a, b)= \lim_\alpha \omega(b\ha(x_\alpha a))= \omega(b\ha(xa)),
$$
by the continuity of $\omega$. This easily implies that $\overline{\vp_\omega} (xa,b)= \overline{\vp_\omega} (a,x\ha b)$, for every $x \in \A$ and $a,b \in \BB$, so that $\overline{\vp_\omega}$
 satisfies ({\sf ips}$_3$).

Let now $x \in L(y)$ and $a,b \in \BB$. Now let $\{x^*_\beta\}$ and $\{y_\alpha\}$ nets in $\BB$, $\tau^*$-converging, respectively, to $x^*$ and $y$ and
such that $\vp_\omega((x^*_\beta - x^*_{\beta'})b,(x^*_\beta - x^*_{\beta'})b)  \to 0$ and $\vp_\omega ((y_\alpha - y_\alpha')a, (y_\alpha - y_\alpha')a \to 0$. Then we get
\begin{align*}
\overline{\vp_\omega}((xy)a,b) &= \omega(b\ha(xy)a)\\
&= \ip{\pi_\omega(xy)\lambda_\omega^0(a)}{\lambda_\omega^0(b)}\\
&=\ip{\pi_\omega(x) \mult \pi_\omega(y)\lambda_\omega^0(a)}{\lambda_\omega^0(b)}\\
&=\ip{ \pi_\omega(y)\lambda_\omega^0(a)}{\pi_\omega(x^*)\lambda_\omega^0(b)}\\
&=\lim_{\alpha,\beta} \ip{ \pi_\omega(y_\alpha)\lambda_\omega^0(a)}{\pi_\omega(x_\beta^*)\lambda_\omega^0(b)}\\
&= \lim_{\alpha,\beta} {\vp_\omega}(y_\alpha a, x_\beta^* b)= \overline{\vp_\omega}(ya, x\ha b).
\end{align*}
Thus, ({\sf ips}$_4$) holds. To complete the proof, we need to show that $\lambda_{\overline{\vp_\omega}}(\BB)$ is dense in the Hilbert space $\H_{\overline{\vp_\omega}}$.
This part of the proof is completely analogous to that given in \cite[Proposition 3.6]{FTT} and we omit it.
\end{proof}

If $\A$ is semi-associative and fully representable,  every continuous representable linear functional $\omega$ comes from a closed ips-form $\overline{\vp_\omega}$, but $\overline{\vp_\omega}$ need not be continuous,
in general, unless  {  more assumptions are made on the topology $\tau$.}

\begin{cor}\label{cor: ipscontinuity}Let $\A[\tau]$ be a fully representable semi-associative *-topological \pa\, with multiplication core $\BB$. Assume that $\A[\tau]$ is a Fr\'echet space.
Then, for every $\omega\in \rcab$, $\overline{\vp_\omega}$ is a continuous ips-form.
\end{cor}

\begin{proof} The map $\overline{\lambda_\omega^0}$ is closed and everywhere defined. The closed graph theorem then implies that $\overline{\lambda_\omega^0}$ is continuous.
The statement follows from Proposition \ref{3.3}.
\end{proof}

Summarizing, we have

\begin{thm}  \label{thm_frimplies semis}
 Let $\A[\tau]$ be a fully representable *-topological \pa\, with multiplication core $\BB$ and unit $e \in \BB$.
Assume that $\A[\tau]$ is a Fr\'echet space and the following conditions hold
\begin{itemize}
\item[({\sf rc})] Every linear functional $\omega$ which is continuous and $\BB$-positive is representable;

\item[({\sf sq})] for every $x \in \A$, there exists a sequence $\{b_n\}\subset \BB$ such that $b_n \stackrel{\tau}{\to} x$ and the sequence $\{b_n^* b_n\}$ is increasing,  in the sense of the order of $\A^+(\BB)$.
\end{itemize}
Then $\A$ is *-semisimple.
\end{thm}
\begin{proof} Assume, on the contrary, that there exists $x\in \A \setminus \{0\}$ such that $\vp(x,x)=0$, for every $\vp \in \pppb$. If $\omega$ is continuous and $\BB$-positive,
then by assumption it is representable; thus $\overline{\vp_\omega}$, which is everywhere defined on $\A\times \A$, is continuous, by Corollary \ref{cor: ipscontinuity}. Let $x=\lim_{n\to \infty} {b_n}$,
with ${b_n} \in \BB$ and $\{b_n^* b_n\}$ increasing. Then we have
$$ 0\leq \lim_{n\to \infty} \omega(b\ha_n b_n)=\lim_{n\to \infty} \vp_\omega (b_n,b_n) =\overline{\vp_\omega}(x,x)=0. $$ Then $\omega(b\ha_n b_n)=0$, for every $n \in {\mathbb N}$.
But this contradicts Theorem \ref{thm_separation}.
\end{proof}
As we have seen in Example \ref{ex_contnonrepr}, condition ({\sf rc}) is not fulfilled in general. To get an example of a situation where this condition is satisfied, it is enough to {replace} in
Example \ref{ex_contnonrepr} the normed partial *-algebra $L^1(I)$ with $L^2(I)$ (which is fully representable, as shown in \cite{FTT}). It is easily seen that both condition ({\sf rc}) and ({\sf sq})
 are satisfied in this case.  {  It has been known  since a long time} that this partial *-algebra is *-semisimple \cite{bt_ellepi}.

\section { Bounded elements in *-semisimple \pa s}
\label{sect-bddelem}

*-Semisimple topological partial *-algebras are characterized by the existence of a sufficient family of ips-forms. This fact was used in \cite{att_2010} and in \cite{FTT} to derive a number of  properties that we want to revisit in this larger framework.

\subsection{$\M$-bounded elements}

 { First we adapt }  to the present case the  definition of {$\M$-bounded} elements given in \cite[Def. 4.9]{att_2010}
for an $\Ao$-regular topological  \pa.

\bedefi \label{def_Mboun}
Let $\A$ be a topological \pa\ with multiplication core $\BB$ and a sufficient family $\M$ of continuous ips-forms with core $\BB$.
An element $x\in\A$ is called \emph{$\mc M$-bounded} if there exists $\gamma_x>0$ such that
$$
|\vp(xa,b)|\leq\gamma_x\, \vp(a,a)^{1/2}\varphi(b,b)^{1/2}, \;\forall\, \vp\in\mathcal M,\, a,b\in\BB\,.
$$
\findefi
An useful characterization of $\M$-bounded elements is given by the following proposition, whose proof is similar to that of  \cite[Proposition 4.10]{att_2010}.
\begin{prop} \label{prop: char Mbounded}
Let $\A[\tau]$ be a topological \pa\ with multiplication core $\BB$. Then, an element
 $x\in\A$ is $\M$-bounded if, and only if, there exists $\gamma_x\in\mathbb R$ such that
  $\varphi(xa,xa)\leq\gamma_x^2 \, \varphi(a,a)$ for all $\varphi\in \M$ and $a\in\BB$.
\label{boundedprop}
\end{prop}

If $x,y$ are $\M$-bounded elements and their weak product $x\wmult y$ exists, then $x\wmult y$ is also $\M$-bounded.

\begin{lemma}
\label{inverse} Let the $\M$-bounded element $x\in\A$  have a strong inverse $x^{-1}$.  {  Then $\pi(x)$ has a strong inverse for every quasi-symmetric *-representation $\pi$.}
\end{lemma}
\begin{proof}Let $x\in\A$ with strong inverse $x^{-1}$, i.e., $x\bullet x^{-1}=x^{-1}\bullet x=e$. Let $\pi$ be a *-representation with $\pi(e)=I$, then
$$
I=\pi(e)=\pi(x\bullet x^{-1})=\pi(xx^{-1})=  { \pi(x) \mult \pi (x^{-1})}=\pi(x)\circ\pi(x^{-1}).
$$
It follows that the strong inverse $\pi(x)^{-1}:=\pi(x^{-1})$ of $\pi(x)$ exists.
\end{proof}

Given $X \in \LDH$,   we denote by  $\rho_\circ^\D(X)$ the set of all complex numbers $\lambda$ such that
{$X-\lambda I_\D$} has a strong bounded inverse \cite[Section 3]{antratsc} and by $\sigma_\circ^\D(X):= {\mb C}\setminus \rho_\circ^\D(X)$
the corresponding spectrum of $X$.

If $\pi$ is a *-representation of $\A$, from \cite[Proposition 3.9]{antratsc}
it follows that $\sigma_\circ^\D(\pi(x))=\sigma(\overline{\pi(x)})$. If, in particular, $\pi$ is a quasi-symmetric *-representation, we can conclude, by Lemma \ref{inverse}, that $\rho^\M(x)\subseteq\rho(\overline{\pi(x)})=\rho_\circ^\D(\pi(x))$, where $\rho^\M(x)$ denotes the set of complex numbers $\lambda$ such that the strong inverse $(x-\lambda e)^{-1}$ exists as an $\M$-bounded element of $\A$
\cite[Definition 4.28]{att_2010}. Hence,
\begin{equation}\label{spectri}
  \sigma(\overline{\pi(x)})\subseteq\sigma^\M(x).
\end{equation}

 { Exactly as}  for partial *-algebras of operators, there is here a natural distinction between hermitian elements $x$ of $\A$ (i.e. $x=x\ha$) and
 {  self-adjoint} elements (hermitian and with real spectrum).

\begin{defn}
The element $x\in \A$ is said  { $\M$-self-adjoint} if it is hermitian and $\sigma^\M(x)\subseteq\mathbb{R}$.
\end{defn}

\begin{prop}

If $x\in\A$ is  { $\M$-self-adjoint}, then for every quasi-symme\-tric $\pi\in\repb$, the operator $\pi(x)$ is essentially self-adjoint. \end{prop}
\begin{proof}
If $x\in\A$ is  { $\M$-self-adjoint}, then, for every $\pi\in\repb$, the operator $\pi(x)$ is symmetric and
$\sigma^\M(x)\subseteq\mathbb{R}$. By \eqref{spectri}
it follows that $\overline{\pi(x)}$ is self-adjoint, hence $\pi(x)$ is essentially   self-adjoint.
\end{proof}

\subsection{Order bounded elements}
\label{subsec-orderbdd}

\subsubsection{Order structure}\label{sect_order}

Let $\A[\tau]$ be a topological \pa\ with multiplication core $\BB$. If $\A[\tau]$ is *-semisimple, there is a natural order on $\A$ defined by the family $\pppb$ or by any sufficient subfamily $\M$ of $\pppb$,
 and this
order  can be used to define a different notion of \emph{boundedness} of an element $x\in \A$ \cite{FTT, schm_weyl,vidav}.

\bedefi \label{def_cone} Let $\A[\tau]$ be a topological \pa\ and $\BB$ a subspace of $R\A$.
A subset $\K$ of $\A_h:=\{x \in \A:\, x=x^*\}$ is called a $\BB$-\emph{admissible wedge} if
\begin{itemize}
\item[(1)] $e \in \K$, if $\A$ has a unit $e$;
\item[(2)] $x+y \in \K, \;\forall\, x,y \in \K$;
\item[(3)] $\lambda x \in \K, \;\forall\, x \in \K, \, \lambda\geq 0$;
\item[(4)] $(a^*x)a=a^*(xa) =:a^*xa\in \K, \;\forall\, x \in \K, \, a \in \BB$.
\end{itemize}
\findefi
\noi As usual, $\K$ defines an order on the real vector space $\A_h$ by $x \leq y \Leftrightarrow y-x \in \K$.

\medskip In the rest of this section, we will suppose that the partial *-algebras under consideration are  \emph { semi-associative}. Under this assumption, the first equality in (4) of
Definition \ref{def_cone} is automatically satisfied.

\medskip Let $\A$ be a topological partial *-algebra with multiplication core $\BB$. We put
$$
\BB^{(2)}=\left\{\sum_{k=1}^n x_k^* x_k, \, x_k \in \BB,\, n \in {\mb N}\right\}.
$$
\textcolor{black}If $\BB$ is a *-algebra, this is nothing but the set (wedge) of positive elements of $\BB$.
The \emph{$\BB$-strongly positive} elements of $\A$ are then defined as the elements of $\A^+(\BB):={\overline{\BB^{(2)}}}^\tau$, {  already defined in Section \ref{sect-represfunct}.}
 Since $\A$ is  { semi-associative}, the set $\A^+(\BB)$ of $\BB$-strongly positive elements is a $\BB$-admissible wedge.°

\medskip
We also define
$$
\A^+_{\rm alg}=\left\{\sum_{k=1}^n x_k^* x_k, \, x_k \in R\A,\, n \in {\mb N}\right\},
$$
the set (wedge) of positive elements of $\A$ and we put
$\A^+_{\rm top}:=\overline{\A^+_{\rm alg}}^\tau$.
The  { semi-associativity}  implies that $R\A \cdot R\A \subseteq R\A$ and then $\A^+_{\rm top}$ is $R\A$-admissible.

\medskip
Let $\M\subseteq \pppb$.
An element $x \in \A$ is called $\M$-\emph{positive} if
$$
 \vp(xa,a)\geq 0, \quad \forall \vp \in \M, \, a\in \BB.
$$
An $\M$-positive element is automatically hermitian.
Indeed, if  $\vp(xa,a)\geq 0$,  \mbox{$\forall\, \vp\in\M$,}  $\forall\, a\in\BB$, then $\vp(a,x\ha a) = \vp(xa,a) \geq 0$ and $\vp(x\ha a,a) \geq 0$;
hence $\vp((x-x\ha)a,a)=0, \, \forall\, \vp\in\M,  \forall\, a\in\BB$. By (iii) of Lemma \ref{lem:Msuff}, it follows that $x=x\ha$.

We denote by $\A_\M^+$ the set of all $\M$-positive elements. Clearly $\A_\M^+$ is a $\BB$-admissible wedge.

\begin{prop}  The following inclusions hold
\begin{equation}\label{eq:3cones}
\A^+(\BB) \subseteq \A^+_{\rm top} \subseteq \A_\M^+, \quad \forall \M \subseteq \pppb.
\end{equation}
\end{prop}

\begin{proof} We only prove the second inclusion. Let $x \in \A^+_{\rm top}$. Then $x =\lim_\alpha b_\alpha$ with $b_\alpha
=\sum_{i=1}^n c_{\alpha,i}^*c_{\alpha,i}$, $c_{\alpha,i}\in R\A$. Thus,
\begin{align*} \vp(xa,a)&= \lim_\alpha \vp(b_\alpha a, a) = \lim_\alpha \vp(\sum_i (c_{\alpha,i}^*c_{\alpha,i})a,a)\\
 & = \lim_\alpha \sum_i\vp( (c_{\alpha,i}^*c_{\alpha,i})a,a)= \lim_\alpha \sum_i \vp(c_{\alpha,i}a, c_{\alpha,i} a)\geq 0.
\end{align*}
by ({\sf ips}$_4$).
\end{proof}

Of course, one expects that under certain conditions the converse inclusions hold, or   { that} the three sets in \eqref{eq:3cones} actually coincide.
A partial answer is given in Corollary \ref{cor_compare}.

\beex We give here two examples where the wedges considered above coincide.

(1) The first example, very elementary, is obtained by considering the space $L^p(X)$, $p\geq 2$. Indeed, it is easily seen that the $\M$-positivity of a function $f$ simply means that $f(t)\geq 0$ a.e. in $X$ ($\M$ is here the family of ips-forms defined in Example \ref{ex_34}). On the other hand it is well-known that such a function can be approximated in norm by a sequence of nonnegative functions of $L^\infty(X)$.

(2) Let $T$ be a self-adjoint operator with dense domain $D(T)$ and denote
by $E(\cdot)$ the spectral measure of $T$. We consider the space $\LDH$ where $D:= \D^\infty(T)= \bigcap_{n\in {\mb N}}D(T^n)$. We prove that if $\LDH$ is endowed
with the topology ${\sf t}_{s^*}$ and $\M$ is the family of ips-forms defined in Example \ref{ex_33}, then every $X\in \LDH$ which is $\M$-positive, i.e., $\ip{X\xi}{\xi}\geq 0$,
 for every $\xi \in \D$, is the ${\sf t}_{s^*}$-limit of elements of $\LBD^{(2)}$.  Indeed, if $\Delta, \Delta'$ are bounded Borel subsets of the real line, then  $E(\Delta ^{(\prime)})\xi \in \D$, for every $\xi \in \H$.
This implies that $E(\Delta')YE(\Delta)$ is a bounded operator in $\H$, for every $Y \in \LDH$
 and its
restriction to $\D$ belongs to $\LBD$. Put $\Delta_N=(-N, N]$, $N \in {\mb N}$ and $\delta_m=(m,m+1]$, $m\in {\mb Z}$. The $\M$-positivity of $X$ implies that $E(\Delta_N)XE(\Delta_N)= B_N^*B_N$ for some bounded operator $B_N$.
Then, observing that $\sum_{m \in {\mb Z}}E(\delta_m)= I$, in strong  or strong*-sense, we obtain
\begin{align*}
 E(\Delta_N)XE(\Delta_N)&= B_N^*B_N= E(\Delta_N)B_N^*B_N E(\Delta_N)\\
&=E(\Delta_N)B_N^* \left(\sum_{m \in {\mb Z}}E(\delta_m)\right) B_NE(\Delta_N)\\
&= \sum_{m \in {\mb Z}}(E(\Delta_N)B_N^* E(\delta_m)) (E(\delta_m)B_N E(\Delta_N)).
\end{align*}
This proves that $E(\Delta_N)XE(\Delta_N)$ belongs to the ${\sf t}_{s^*}$-closure of $\LBD^{(2)}$
Now, if we let $N \to \infty$, we easily get $\|X\xi - E(\Delta_N)XE(\Delta_N)\xi\|\to 0$ and so the statement is proved.
\enex

\medskip An improvement of Theorem \ref{thm_separation} is provided by the following

\begin{cor} \label{thm_separation2} Let $\M$ be sufficient. Then,  for every $x \in \A^+_\M$, $x \neq 0$, there exists $\omega\in \rcab$ with the properties
\begin{itemize}
  \item[(a)] $\omega (y)\geq 0$, $\; \forall y \in \A^+_\M$;
  \item[(b)] $\omega(x)>0$.
\end{itemize}
\end{cor}
\begin{proof} By the previous proposition, if $x \in \A^+_\M$, $x \neq 0$, there exist $\vp \in \M$ and $a \in \BB$ such that $\vp(xa,a)>0$. Hence the linear functional $\omega(y):=\vp(ya,a)$ has the
desired properties.
\end{proof}

\medskip
\begin{prop}\label{prop_Mcone}Let the family $\M$ be sufficient. Then, $\A^+_\M$ is a cone, i.e.,  $\A^+_\M \cap (-\A^+_\M)=\{0\}$.
\end{prop}
\begin{proof} If $x\in \A^+_\M \cap (-\A^+_\M)$, then $\vp(xa,a)\geq 0$ and $\vp((-x)a,a)\geq 0$, for every $\vp \in \M$ and $a \in \BB$.
Hence $\vp(xa,a)= 0$, for every $\vp \in \M$ and $a \in \BB$. The sufficiency of $\M$ then implies $x=0$.
\end{proof}

\berem The fact that  $\A^+_\M$ is a cone automatically implies that $\A^+(\BB)$ is a cone too. \enrem

\medskip The following statement shows that $\pppb$-positivity is exactly what is needed if we want the order to be preserved under any continuous *-representation.
A partially equivalent statement is given in \cite[Proposition 3.1]{FTT}.  { For making the notations lighter, we put $\A_{\mc P}^+:=\A_{\pppb}^+$.}

\begin{prop}\label{prop_pos}
Let $\A$ be a topological partial *-algebra with multiplication core $\BB$ and unit $e \in \BB$. Then,  { the element $x\in\A$ belongs to $\A^+_\P$ if and only if  the operator $\pi(x)$ is positive for every
 $(\tau, {\sf t}_s)$-continuous *-representation $\pi$ with $\pi(e)=I_{\D(\pi)}$.}
\end{prop}
\begin{proof}
Let $x\in\A^+_\P$ and  { let $\pi$ be} a $(\tau, {\sf t}_s)$-continuous *-representation of $\A$ with $\pi(e)=I_{\D(\pi)}$.
The sesquilinear form $\vp_\pi^\xi$, defined by
$$
\vp_\pi^\xi(x,y):=\ip{\pi(x)\xi}{\pi(y)\xi}, \quad x,y \in \A,
$$
is a continuous ips-form as shown in the proof of Proposition \ref{prop_finalnew}.
Then,
$$
\vp_\pi^\xi(xa,a)=\ip{\pi(xa)\xi}{\pi(a)\xi}=\ip{(\pi(x)\mult\pi(a))\xi}{\pi(a)\xi};
$$
in particular, for $a=e$, $\ip{\pi(x)\xi}{\xi}\geq0$.

Conversely, let  $\vp\in\P $ and $\pi_\vp$ the corresponding GNS representation.
{   Then, as remarked in the proof of Proposition \ref{prop_finalnew}, $\pi_\vp$ is $(\tau, {\sf t}_{s})$-continuous.}
We have, for every $a \in \BB$,
$$
 \vp(xa,a)=\ip{\pi_\vp(x)\lambda_\vp(a)}{\lambda_\vp(a)}\geq 0,
 $$
 i.e., $x\in\A^+_\P.$
\end{proof}

\begin{prop}\label{prop_new_fr}
Let $\A$ be a fully-representable  *-topological partial *-algebra with multiplication core $\BB$ and unit $e \in \BB$. Assume that $\A[\tau]$ is a Fr\'echet space. Then the following statements are equivalent:
\begin{itemize}
  \item[(i)] $x\in\A_\P^+$;
  \item[(ii)]$\omega(x)\geq0$, $\forall \omega\in\R_c(\A,\BB)$.
\end{itemize}
\end{prop}
\begin{proof}(i)$\Rightarrow$(ii): If  $x\in \A_\P^+$, then $\vp(xa,a)\geq0$, $\forall\vp\in\pppb$, $\forall a\in\Ao$. If $\omega\in\R_c(\A,\BB)$, by the assumptions and
by Proposition \ref{prop: 3.10} it follows that $\overline{\vp_\omega}$ is an everywhere defined ips-form and thus, by Corollary \ref{cor: ipscontinuity}, it is continuous. Hence,
$$
\omega(a^*xa)=\overline{\vp_\omega}(xa,a)\geq 0,\quad\forall a\in\BB.
$$
For $a=e$, we get that $\omega(x)\geq0$.

(ii)$\Rightarrow$(i): If $\omega(x)\geq0$, $\forall\omega\in\R_c(\A,\BB)$, then this also holds  for every linear functional $\omega_\vp^a$, $a \in \BB$, defined by $\vp\in\pppb$ as in
Proposition \ref{3.3}. Then
$$
\vp(xa,a)=\omega_\vp^a(x)\geq0,\; \forall\, \vp\in\pppb.
$$
By definition, this means that $x\in\A_\P^+.$
\end{proof}

In complete analogy with Proposition 3.9 of \cite{FTT}, one can prove the following

\begin{prop}\label{prop_ftt} Let $\A[\tau]$  be a *-semisimple *-topological partial *-algebra with multiplication core $\BB$.

Assume that the following condition \mbox{\sf{(P)}} holds:
\begin{align*}
\textrm{\sf (P)}
\,&\mbox{ $y \in \A$ and $\omega(a^*ya)\geq 0$, for every $\omega \in {\mc R}_c(\A,\BB)$ and $a \in \Ao$,}\\  &\mbox{imply $y \in  { \A^+(\BB)}$}.
\end{align*}
Then, for an element $x \in \A$, the following statements  are
equivalent:
\begin{itemize}
             \item[(i)] $x\in \A^+(\BB)$;
             \item[(ii)]$\omega(x)\geq 0$, for every $\omega \in {\mc R}_c(\A,\BB)$;
             \item[(iii)]$\pi(x)\geq 0$, for every $(\tau, {\sf t}_w)$-continuous $^*$-representation $\pi$ of $\A$.
           \end{itemize}

\end{prop}

\berem In \cite[Proposition 3.9]{FTT} it was required that the family ${\mc R}_c(\A,\Ao)$ of continuous linear functionals does not annihilate positive elements. This is always true for *-semisimple \pa s,
because of Proposition \ref{prop_Mcone}.
\enrem

The previous propositions allow to compare the different cones defined so far.
\begin{cor}\label{cor_compare} Under the assumptions of Propositions \ref{prop_new_fr} and \ref{prop_ftt}, one has $ \A^+(\BB) = \A^+_{\P}$.
\end{cor}

\subsubsection{Order bounded elements}\label{sect_5}

Let $\A[\tau]$ be a topological partial *-algebra with multiplication core $\BB$ and  unit $e\in \BB$. As we have seen in Section \ref{sect_order}, $\A[\tau]$ has several natural orders, all related to the topology $\tau$.
Each one of them
can be used to define \emph{bounded} elements. We begin in a purely algebraic way  starting from an arbitrary $\BB$-admissible cone $\K$.

Let $x \in \A$; put $\Re (x) =\frac{1}{2}(x+x^*)$, $\Im (x)= \frac{1}{2i}(x-x^*)$. Then $\Re{x}, \Im(x) \in \A_h$ (the set of  { self-adjoint} elements of $\A$) and $x= \Re(x)+i\Im (x)$.

\bedefi An element $x\in \A$ is called $\K$-\emph{bounded} if there exists $\gamma \geq 0$ such that
$$
 \pm \Re(x) \leq \gamma e; \qquad \pm \Im(x)\leq \gamma e.
 $$
We denote by $\A_b(\K)$ the family of $\K$-bounded elements.
\findefi

The following statements are easily checked.
\begin{itemize}
\item[(1)]$ \alpha x+\beta y \in \A_b(\K), \quad \forall x,y \in \A_b(\K), \, \alpha, \beta \in {\mb C}$.
\item[(2)]$x \in \A_b(\K) \Leftrightarrow x^* \in \A_b(\K)$.
\end{itemize}
\berem
If $\A$ is a *-algebra then, as shown in \cite[Lemma 2.1]{schm_weyl}, one also has
\begin{itemize}
\item[(3)]$x, y\in \A_b(\K) \Rightarrow xy\in \A_b(\K)$.
\item[(4)]$a\in \A_b(\K)  \Leftrightarrow aa^* \in \A_b(\K)$.
\end{itemize}
These statements do not hold in general when $\A$ is a \pa. They are true, of course, for elements of $\BB$.
\enrem

For $x\in \A_h$, put
$$
\|x\|_b:= \inf\{\gamma>0:\, -\gamma e \leq x \leq \gamma e\}.
$$
$\|\cdot\|_b$ is a seminorm on the real vector space $(\A_b(\K))_h$.

\begin{lemma} Let $\M$ be sufficient. If $\K=\A^+_\M$, then $\|\cdot\|_b$ is a norm on $(\A_b(\M))_h$.
\end{lemma}
\begin{proof} By Proposition \ref{prop_Mcone}, $\A^+_\M$ is a cone.  Put $E=\{\gamma>0:\, -\gamma e \leq x \leq \gamma e\}$. If $\inf E=0$, then, for every $\epsilon >0$, there exists $\gamma_\epsilon\in E$ such that $\gamma_\epsilon <\epsilon$. This implies that $-\epsilon e \leq x\leq \epsilon e$. If $\vp \in \M$, we get
$-\epsilon \,\vp(a,a) \leq \vp(xa,a)\leq \epsilon\, \vp(a,a)$, for every $a \in \BB$. Hence, $\vp(xa,a)=0$. By the sufficiency of $\M$, it follows that $x=0$.
\end{proof}

\medskip Let $\A[\tau]$ be a *-semisimple  topological partial *-algebra with multiplication core $\BB$. We can then specify the wedge $\K$ as one of those defined above.  Take first $\K=\A^+_\M$,
where $\M=\pppb$ is the sufficient family of all continuous  { ips-forms}  with core $\BB$.
For simplicity, we write  { again} $\P:=\pppb$, hence $\A^+_\P:=\A^+_\pppb$ and  $\A_b(\P):=\A_b(\pppb)$.

\begin{prop}\label{prop_53} If $x \in \A_b(\P)$, then $\pi(x)$ is a bounded operator, for every  $(\tau, {\sf t}_s)$-continuous *-representation of $\A$.
Moreover, if $x=x^*$, $\|\pi(x)\|\leq \|x\|_b$.
\end{prop}

\begin{proof} This follows easily from Proposition \ref{prop_pos} and from the definitions.
\end{proof}

The following theorem generalizes \cite[Theorem 5.5 ]{FTT}.
\begin{thm}\label{thm_420}
Let $\A[\tau]$ be a fully representable,  semi-associative *-topo\-lo\-gi\-cal \pa,  with multiplication core $\BB$ and unit $e \in \BB$. Assume that $\A[\tau]$ is a Fr\'echet space.
Then the following statements are equivalent:
\begin{itemize}
  \item[(i)]$x\in\A_b(\P)$.
  \item[(ii)]There exists $\gamma_x>0$ such that
  $$
  |\omega(a^*xa)|\leq\gamma_x\omega(a^*a),\,\forall\,\omega\in\R_c(\A,\BB),\forall\, a\in\BB.
  $$
  \item[(iii)]There exists $\gamma_x>0$ such that
  $$
  |\omega(b^*xa)|\leq\gamma_x\omega(a^*a)^{1/2}\omega(b^*b)^{1/2},\,\forall\,\omega\in\R_c(\A,\BB),\forall\, a,b\in\BB.
  $$
\end{itemize}
\end{thm}

\begin{proof} It is sufficient to consider the case $x=x^*$.

(i) $\Rightarrow$ (iii) If $x=x^* \in\A_b(\P)$, there exists
$\gamma>0$ such that $-\gamma e \leq x \leq \gamma e$; or, equivalently,
$$ -\gamma \vp(a,a) \leq \vp(xa,a) \leq \gamma \vp(a,a), \; \forall\, \vp \in \P, \, a\in \BB.$$
Since $\A$ is fully representable, $D( \overline{\vp_\omega})=\A$ and, by Corollary \ref{cor: ipscontinuity}, it is a continuous  { ips-form} with core $\BB$. Thus, as seen in the proof of Proposition \ref{prop_finalnew}, $\pi_{\overline{\vp_\omega}}$ is $(\tau,{\sf t}_s)$-continuous. Hence, by Proposition \ref{prop_53}, $\pi_{\overline{\vp_\omega}}(x)$ is bounded and $\|\pi_{\overline{\vp_\omega}}(x)\|\leq \|x\|_b$. Therefore,
\begin{align*}
|\omega(b^*xa)|
&=|\overline{\vp_\omega(xa,b)}|\leq \overline{\vp_\omega}(xa,xa)^{1/2} \vp_\omega(b,b)^{1/2}\\
&= \|\pi_{\overline{\vp_\omega}}(x)\lambda_{\overline{\vp_\omega}}(a)\| \, \vp_\omega(b,b)^{1/2} \leq \|x\|_b \gamma_x \omega(a\ha a)^{1/2} \omega (b\ha b)^{1/2}.
\end{align*}

(iii) $\Rightarrow$ (ii) is obvious.

 (ii) $\Rightarrow$ (i) Assume now that there
exists $\gamma_x >0$ such that \begin{equation}\label{eq: bound}|\omega(a^*xa)|\leq \gamma_x
\omega(a^*a), \quad \forall \ \omega \in {\mc R}_c(\A,\BB), \,
a\in \BB .\end{equation} Define
$$
\widetilde{\gamma}:=\sup\{| \omega(a^*xa)|:
\omega \in {\mc R}_c(\A,\Ao), \, a \in \Ao, \omega(a^*a)=1 \}.
$$
Let $\vp \in \P$ and $a \in \BB$. By  Proposition \ref{3.3}, the linear functional $\omega_\vp^a$ defined by $\omega_\vp^a(x) = \vp(xa,a)$, $x \in \A$, is continuous and representable.
If $\vp(a,a)=0$, then, by \eqref{eq: bound}, $\vp(xa,a)=0$. If  $\vp(a,a)>0$,  we get
$$
\vp((\widetilde{\gamma} \, e\pm x)a,a) =\widetilde{\gamma}\,\vp(a,a) \pm \vp(xa,a)=\vp(a,a) (\widetilde{\gamma} \pm \vp(xu,u)) \geq 0,
$$
where $u = {a}{\vp(a,a)^{-1/2}}$.
Hence,  { by the} arbitrariness of $\vp$ and $a$,  { we have} $x \in \A_b(\P)$.
\end{proof}

We can now compare the notion of order bounded element with that of $\pppb$-bounded element given in Definition \ref{def_Mboun}.

\begin{thm}\label{theorem_equiv} Let $\A[\tau]$ be a *-semisimple topological partial *-algebra with multiplication core $\BB$ and unit $e\in \BB$.
 For $x \in \A$, the following statements are equivalent.
\begin{itemize}
\item[(i)] $x$ is $\pppb$-bounded.
\item[(ii)]  $x \in \A_b(\P)$.
 \item[(iii)]$\pi(x)$ is bounded, for every $\pi \in \rep$, and
    $$\sup\{\|\overline{\pi(x)}\|, \, \pi \in \rep \}<\infty.
    $$
\end{itemize}
\end{thm}
 \begin{proof} It is sufficient to consider the case $x=x^*$.

(i) $\Rightarrow$ (ii): If $x=x^*$ is $\pppb$-bounded, we have, for some $\gamma>0$,
$$
 -\gamma \vp(a,a) \leq \vp(xa,a)\leq \gamma\vp(a,a), \; \forall \,\vp \in \P, a \in \BB.
$$
This means that $-\gamma e \leq x \leq \gamma e$ in the sense of the order induced by $\A^+_\P$. Hence $x \in \A_b(\P)$.

(ii) $\Rightarrow$ (iii): Let $\pi \in \rep$ and $\xi \in \D(\pi)$. Define $\vp_\pi^\xi$ as in the proof of Proposition \ref{prop_pos}. Then $\vp_\pi^\xi \in \P$.
Hence by (ii), $|\vp_\pi^\xi(xa,a)|\leq \gamma_x \vp_\pi^\xi(a,a)$, for some $\gamma_x>0$ which depends on $x$ only. In other words, $|\ip{\pi(x)\xi}{\xi}|\leq \gamma_x \|\xi\|^2$.
This in turn easily implies that $|\ip{\pi(x)\xi}{\eta}|\leq \gamma_x \|\xi\|\|\eta\|$, for every $\xi, \eta \in \D(\pi)$. Hence $\pi(x)$ is bounded and $\|\overline{\pi(x)}\|\leq \gamma_x$.

 (iii) $\Rightarrow$ (i): Put $\gamma_x:= \sup\{\|\overline{\pi(x)}\|, \, \pi \in \rep \}$. Then
 $$
  |\ip{\pi(x)\xi}{\xi}|\leq \|\pi(x)\xi\| \|\xi\|  \leq \gamma_x \|\xi\|^2, \; \forall\, \xi \in \D_\pi.
  $$
This in particular holds for the GNS representation $\pi_\vp$ associated to any $\vp \in \pppb$, {   since $\pi_\vp$ is $(\tau, {\sf t}_s)$-continuous}.
Hence, for every $a \in \BB$, we get
$$|\vp(xa,a)| = |\ip{\pi_\vp(x) \lambda_\vp(a)}{\lambda_\vp(a)}|\leq \gamma_x \| \lambda_\vp(a)\|^2= \gamma_x \vp(a,a).$$
Using the polarization identity, one finally gets
$$
|\vp(xa,b)|\leq\gamma_x\, \vp(a,a)^{1/2}\varphi(b,b)^{1/2}, \;\forall\, \vp\in\P,\, a,b\in\BB.
$$
This proves that $x$ is $\pppb$-bounded.
\end{proof}

Theorem \ref{theorem_equiv} shows that, under the assumptions we have made, order boundedness is nothing but the $\M$-boundedness studied in \cite{att_2010}.
So all results proved there apply to the present situation (in particular those concerning the structure of the topological \pa\ under consideration and its spectral properties).
Clearly, the crucial assumption is the existence of sufficiently many continuous ips-forms, that is, the *-semisimplicity.

 \beex In particular, Theorem \ref{theorem_equiv} shows that, in $\LDH$ (see Example \ref{ex_33} for notations), bounded elements defined by $\M$ and those defined by the order
coincide and (as expected) the family of bounded elements is $\LBDH$. Of course, one could get this result directly, using well-known properties of operators.

Also in the case of $L^p$-spaces ($p>2$) considered in Example \ref{ex_34}, one obtains that the two notions of boundedness coincide and the bounded part is exactly $L^\infty(X)$,
as can also be proved by elementary arguments.

\enex

\medskip So far we have considered the order boundedness defined by the cone $\A^+_{\P}$, but other choices are possible. For instance we may consider the order induced by $\A^+(\BB)$.
It is clear that if $x \in \A_b (\A^+(\BB))$ then $x \in \A_b({\P})$. On the other hand, if $x \in \A_b({\P})$ and the assumptions of Theorem \ref{thm_420} hold,
there exists $\gamma_x>0$ such that
$$
|\omega(a^*xa)|\leq\gamma_x\omega(a^*a),\,\forall\,\omega\in\R_c(\A,\BB),\forall\, a\in\BB.
$$
 { Hence, if condition \sf (P)} holds too,
 we can conclude, by adapting the argument used in the proof of Theorem \ref{thm_420}, that $x \in \A_b (\A^+(\BB))$. We leave a deeper analysis of the general question to future papers.

\section{Concluding remarks}

As we have discussed in the Introduction, the notion of bounded element for a topological partial *-algebra plays an important  role for the whole discussion.
We have at hand two different notions,  one ($\M$-boundedness)  based on a sufficient family of ips-forms, and another one (order boundedness) based on some
 $\BB$-admissible wedge, where $\BB$ is a multiplication core. Both seem very reasonable definitions and, as we have seen, they  can be compared    in many occasions.
In the framework of (topological) \mbox{*-algebras}, it is even possible that every element is order bounded (see examples in \cite[Section 5]{schmd_alg_geom}).\footnote{
The terminology adopted in that paper comes from algebraic geometry, so that an admissible cone is called there a \emph{quadratic module}.}
 The analogous situation for partial *-algebras is unsolved (in other words we do not know if there exist topological partial \mbox{*-algebras} where every element
 is order bounded) and we conjecture that a  \emph{complete} topological partial *-algebra $\A$ whose elements are all bounded is necessarily an algebra.
 This is certainly true in the case where $\M$-boundedness is considered, where $\M$ is a well-behaved family of ips-forms in the sense of
Definition 4.26 of \cite{att_2010}. Indeed, as shown there (Proposition 4.27), under these assumptions the set of $\M$-bounded elements is a C*-algebra.
The same, of course, holds true in the situation considered in Theorem \ref{theorem_equiv}, if the family $\pppb$ is well-behaved.
However, the general question is  open.


\begin{thebibliography}{100}

\bibitem{ait_book} J-P. Antoine, A. Inoue and C. Trapani, \textit{Partial *-algebras and their operator realizations}, Kluwer, Dordrecht, 2002.

\bibitem{att_continuoushomom} {J-P. Antoine, C. Trapani and F. Tschinke}, \textit{Continuous *-homomorphisms of Banach
Partial *-algebras},  Mediterr. j. math. {4}  (2007), 357--373.

\bibitem{antratsc} {J-P. Antoine, C. Trapani and F. Tschinke}, \textit{Spectral properties  of partial *-algebras}   Mediterr. j. math. {7}  (2010), 123--142.

\bibitem{att_2010}J-P. Antoine, C.Trapani and F. Tschinke, \textit{Bounded elements in certain topological partial *-algebras}, Studia Math. {203} (2011), 223--251.

\bibitem{bit_reprfunct} F. Bagarello, A. Inoue and C.Trapani, \textit{Representable linear functionals on partial *-algebras}, Mediterr. j. math. {9} (2012), 153-163.

\bibitem{bt_ellepi} F. Bagarello and C.Trapani, \textit{$L^p$-spaces as quasi *-algebras}, J. Math. Anal. Appl . {197} (1996), 810-824.

\bibitem{FTT}M. Fragoulopoulou, C. Trapani and S. Triolo, \textit{Locally convex quasi *-algebras with sufficiently many *-representations},
J. Math. Anal. Appl. 388 (2012), 1180-1193.

\bibitem{schmu} K. Schm\"udgen, \textit{Unbounded operator algebras and representation theory}, Birkh\"auser Verlag, Basel, 1990.

\bibitem{schm_weyl} {K. Schm\"udgen}, \textit{A strict Positivstellensatz for theWeyl algebra}, Math. Ann. {331} (2005), 779--794.

\bibitem{schmd_alg_geom} {K. Schm\"udgen}, \textit{Noncommutative real algebraic geometry -
Some basic concepts and first ideas}, in \textit{Emerging Applications in  Algebraic Geometry}, ed. by M. Putinar and S. Sullivant, Springer, 2009.

\bibitem{ct_ban} C. Trapani, \textit{*-Representations, seminorms and structure properties of normed *-algebras},
Studia Mathematica, Vol. 186, 47-75 (2008).

\bibitem{tratsc} {C. Trapani and F. Tschinke}, \textit{Unbounded C*-seminorms and biweights on
partial *-algebras} Mediterr. j. math. {2} (2005) 301--313.

\bibitem{vidav}{I. Vidav}, \textit{On some *-regular rings}, Acad. Serbe Sci. Publ. Inst. Math. {13} (1959) 73--80.
\end{thebibliography}
\end{document}